\documentclass[12pt]{amsart}
\usepackage{amsmath}
\usepackage{amssymb}
\usepackage{enumitem}
\usepackage{color}

\usepackage{hyperref}
\hypersetup{
     colorlinks   = true,
     citecolor    = black,
     linkcolor    = blue
} 
\usepackage{cite}
\usepackage{url}
\usepackage[usenames,dvipsnames]{xcolor}
\usepackage{verbatim}

\usepackage{graphicx}
\usepackage{tikz}
\usepackage{tikz-cd}
\usepackage{caption}
\usepackage{subcaption}

\usepackage{float}

\usepackage{booktabs}
\usepackage{array}

\makeatletter 
\def\@cite#1#2{{\m@th\upshape\bfseries%
[{#1\if@tempswa{\m@th\upshape\mdseries, #2}\fi}]}}
\makeatother 

\theoremstyle{plain}
\newtheorem{theorem}{Theorem}
\newtheorem{coro}{Corollary}
\newtheorem{thm}{Theorem}[section]
\newtheorem{cor}[thm]{Corollary}
\newtheorem{prop}[thm]{Proposition}
\newtheorem{lem}[thm]{Lemma}
\newtheorem{lemma}[thm]{Lemma}

\theoremstyle{definition}
\newtheorem{defn}[thm]{Definition}
\newtheorem{ass}[thm]{Assumption}
\theoremstyle{}
\newtheorem{rem}[thm]{Remark}

\numberwithin{equation}{subsection}
\renewcommand{\bold}[1]{\medskip \noindent {\bf #1 }\nopagebreak}

\newcommand{\nc}{\newcommand}
\newcommand{\rnc}{\renewcommand}

\nc\bA{\mathbb{A}}
\nc\bB{\mathbb{B}}
\nc\bC{\mathbb{C}}
\nc\bD{\mathbb{D}}
\nc\bE{\mathbb{E}}
\nc\bF{\mathbb{F}}
\nc\bG{\mathbb{G}}
\nc\bH{\mathbb{H}}
\nc\bI{\mathbb{I}}
\nc{\bJ}{\mathbb{J}} 
\nc\bK{\mathbb{K}}
\nc\bL{\mathbb{L}}
\nc\bM{\mathbb{M}}
\nc\bN{\mathbb{N}}
\nc\bO{\mathbb{O}}
\nc\bP{\mathbb{P}}
\nc\bQ{\mathbb{Q}}
\nc\bR{\mathbb{R}}
\nc\bS{\mathbb{S}}
\nc\bT{\mathbb{T}}
\nc\bU{\mathbb{U}}
\nc\bV{\mathbb{V}}
\nc\bW{\mathbb{W}}
\nc\bY{\mathbb{Y}}
\nc\bX{\mathbb{X}}
\nc\bZ{\mathbb{Z}}
\nc\cA{\mathcal{A}}
\nc\cB{\mathcal{B}}
\nc\cC{\mathcal{C}}
\nc\cD{\mathcal{D}}
\nc\cE{\mathcal{E}}
\nc\cF{\mathcal{F}}
\nc\cG{\mathcal{G}}
\nc\cH{\mathcal{H}}
\nc\cI{\mathcal{I}}
\nc{\cJ}{\mathcal{J}} 
\nc\cK{\mathcal{K}}
\nc\cL{\mathcal{L}}
\nc\cM{\mathcal{M}}
\nc\cN{\mathcal{N}}
\nc\cO{\mathcal{O}}
\nc\cP{\mathcal{P}}
\nc\cQ{\mathcal{Q}}
\nc\cR{\mathcal{R}}
\nc\cS{\mathcal{S}}
\nc\cT{\mathcal{T}}
\nc\cU{\mathcal{U}}
\nc\cV{\mathcal{V}}
\nc\cW{\mathcal{W}}
\nc\cY{\mathcal{Y}}
\nc\cX{\mathcal{X}}
\nc\cZ{\mathcal{Z}}

\newcommand{\ra}{\longrightarrow}
\newcommand{\M}{\mathcal{M}}

\newcommand{\strata}{\mathcal{H}}

\newcommand{\C}{\mathbb C}
\newcommand{\R}{\mathbb R}
\newcommand{\Q}{\mathbb Q}
\newcommand{\Z}{\mathbb Z}
\newcommand{\wt}{\widetilde}
\newcommand{\ds}{\displaystyle}

\newcommand{\Xmin}{(X_{\operatorname{min}}, \omega_{\operatorname{min}})}
\newcommand{\Qmin}{(Q_{\operatorname{min}}, q_{\operatorname{min}})}
\newcommand{\piX}{\pi_{X_{\operatorname{min}}}}
\newcommand{\piQ}{\pi_{Q_{\operatorname{min}}}}

\nc{\dmo}{\DeclareMathOperator}
\rnc{\Re}{\operatorname{Re}}
\rnc{\Im}{\operatorname{Im}}
\rnc{\span}{\operatorname{span}}
\dmo{\rank}{rank}
\dmo{\End}{End}
\dmo{\Hom}{Hom}
\dmo{\Jac}{Jac}
\dmo{\Id}{Id}
\dmo{\Ann}{Ann}
\dmo{\Area}{Area}
\dmo{\CP}{\bC P^1}
\dmo{\Aut}{Aut}
\dmo{\Sp}{Sp}
\dmo{\SL}{SL}
\dmo{\GL}{GL}

\title[Periodic Points in Genus Two]{Periodic Points in Genus Two: Holomorphic Sections over Hilbert Modular Varieties, Teichm\"uller dynamics, and Billiards}
\author[Apisa]{Paul~Apisa}
\begin{document}
\maketitle

\begin{abstract}
We show that all  $\GL(2, \bR)$ equivariant point markings over orbit closures of primitive genus two translation surfaces arise from marking pairs of points exchanged by the hyperelliptic involution, Weierstrass points, or the golden points on $\Omega E_5$. As corollaries, we classify the holomorphically varying families of points over orbifold covers of genus two Hilbert modular surfaces, solve the finite blocking problem on genus two translation surfaces, and show that there is at most one nonarithmetic rank two orbit closure in $\mathcal{H}(6)$. 
\end{abstract}

%
%
%
%

\section{Introduction}\label{S:intro}

The bundle of holomorphic one-forms over the moduli space of Riemann surfaces admits a $\GL(2, \R)$ action arising from its complex structure and Teichm\"uller geodesic flow. While the goal of classifying every $\GL(2, \R)$ orbit closure is elusive in general, McMullen discovered a solution in genus two. In a sequence of papers~\cite{Mc}, \cite{McM:spin}, \cite{Mc4}, and \cite{Mc5}, he showed that every orbit closure is one of the following: a stratum, a locus of torus covers, or a locus of eigenforms of real multiplication by an order in a real quadratic field.

Given a translation surface $(X, \omega)$ a point $p$ is said to be a periodic point if $p$ is not a zero of $\omega$ and $(X, \omega)$ and $(X, \omega; p)$ have orbit closures of the same dimension (that the orbit closures are manifolds and hence have a well-defined dimension follows from work of Eskin-Mirzakhani~\cite{EM} and Eskin-Mirzakhani-Mohammadi~\cite{EMM}). A translation surface $(X, \omega)$ is said to be primitive if $\omega$ is not the pullback of a holomorphic one-form on a lower genus Riemann surface under a holomorphic branched cover. M\"oller~\cite{M2} showed that the only periodic points on primitive Veech surfaces in genus two are Weierstrass points. In Apisa~\cite{Apisa-mp1}, the author showed that the same holds for any genus two translation surface with dense orbit. 

From these data one might be tempted to conjecture that Weierstrass points are the only possible periodic points on primitive genus two translation surfaces. However, Wright discovered two points - dubbed the golden points - on the golden eigenform locus that he conjectured were periodic. This conjecture was confirmed in Kumar-Mukamel~\cite{MK16} and will be re-established by different means in forthcoming work of Eskin-McMullen-Mukamel-Wright. We show the following:

\begin{theorem}\label{T2}
If a primitive genus two translation surface contains a periodic point that is not a Weierstrass point then its orbit closure is the golden eigenform locus and the periodic point is one of the golden points.
\end{theorem}

\noindent The strategy of the proof is to reduce the problem to M\"oller's classification of periodic points on primitive genus two Veech surfaces~\cite{M2} using degeneration arguments and the Mirzakhani-Wright partial compactification of an affine invariant submanifold~\cite{MirWri}. The argument is presented in Section~\ref{S:PT2}.

In fact, Theorem~\ref{T2} implies a stronger result. Say that a collection of points $B$ on a translation surface $(X, \omega)$ is generic if the complex dimension of the orbit closure of $(X, \omega; B)$ is $|B|$ more than the dimension of the orbit closure of $(X, \omega)$. A collection of points will be called nongeneric otherwise.

\begin{coro}\label{C1}
Any nongeneric collection of points on a primitive genus two translation surface must contain singularities of the flat metric, periodic points (i.e. Weierstrass points or golden points), or two points exchanged by the hyperelliptic involution.
\end{coro}
\begin{proof}[Proof of Corollary~\ref{C1} assuming Theorem~\ref{T2}:]
By Apisa-Wright~\cite[Theorem 1.3]{Apisa-Wright} any nongeneric collection of points on a translation surface $(X, \omega)$ that does not contain a singularity or a periodic point must contain two points $p$ and $q$ that are identified under a holomorphic map $f: X \ra Y$ where $Y$ is a lower genus Riemann surface with an abelian or quadratic differential that pulls back to the flat structure on $(X, \omega)$. Moreover, if this map does not factor through a map to a translation surface it is degree two. Since $(X, \omega)$ is primitive it is not a torus cover so $f$ must be a degree two map to the sphere. Therefore, $f$ is the quotient by the hyperelliptic involution.
\end{proof}

As one application of the Theorem~\ref{T2}, we classify the nonarithmetic rank two affine invariant submanifolds in $\mathcal{H}(6)$, using degenerations to find genus two nonarithmetic eigenforms with marked points in the boundary. 

\begin{theorem}\label{T:MWE}
There is at most one nonarithmetic rank two orbit closure in $\mathcal{H}(6)$. It has no periodic points.
\end{theorem}

\noindent The definition of the orbit closure and the proof of Theorem~\ref{T:MWE} will be given in Section~\ref{S:MWE}. This orbit closure will be shown to be a rank two nonarithmetic affine invariant submanifold in forthcoming work of Eskin-McMullen-Mukamel-Wright.

As another application of Theorem~\ref{T2} we compute holomorphic sections of the universal curve restricted to genus two Hilbert modular varieties. 

The locus of principally polarized abelian varieties with real multiplication by an order in a totally real number field - or Hilbert modular variety - has an explicit description as the quotient of $g$ copies of the hyperbolic plane by a lattice (for details see van der Geer~\cite{vdG}). Let $X_D$ be the locus of abelian surfaces with real multiplication by the order $\mathcal{O}_D = \Z [x] / (x^2 + bx + c)$ where $b^2 - 4c = D$. Let $E_D$ be the collection of Riemann surfaces whose Jacobian belong to $X_D$ and let $\Omega E_D$ be the collection of abelian differentials that are $\mathcal{O}_D$-eigenforms on a surface in $E_D$. The golden eigenform locus is defined to be the intersection of $\Omega E_5$ and $\mathcal{H}(1,1)$. 

For any Riemann surface in $E_D$ there are two holomorphic one-forms that are eigenforms of real-multiplication. The collection of (at most four) points where some eigenform vanishes forms a holomorphically varying family of points over $E_D$. Another such collection is the set of Weierstrass points. The question we take up here is whether there are any other such families of points.

Given a finite index subgroup of the mapping class group, let $E_D(\Gamma)$ be the preimage of $E_D$ under the finite orbifold cover of moduli space prescribed by $\Gamma$. If $\Gamma$ is torsion free, then the universal curve on the cover restricts to a holomorphic surface bundle $\pi: C_D(\Gamma) \ra E_D(\Gamma)$.  We will assume in the sequel that $D$ is not a perfect square. 

\begin{theorem}\label{T1}
Let $\Gamma$ be a torsionfree finite index subgroup of the mapping class group. There is a holomorphic section of $\pi$ over some component of $E_D(\Gamma)$ that does not exclusively mark zeros of the eigenforms or Weierstrass points only if $D = 5$.
\end{theorem}
\begin{proof}[Proof of Theorem~\ref{T1} given Theorem~\ref{T2}]:
Suppose that $\pi$ admits a section over a component of $E_D(\Gamma)$ and let $X$ be a Riemann surface in that component. Let $\omega$ be an $\mathcal{O}_D$-eigenform on $X$.  By Apisa~\cite{Apisa-mp1} holomorphic sections of $C_D(\Gamma) \ra E_D(\Gamma)$ must mark periodic points or zeros of $(X, \omega)$. By Theorem~\ref{T2}, periodic points that are not Weierstrass points only exist if $D = 5$ and the points are the golden points. 
\end{proof}


Finally, we will show that Theorem~\ref{T2} resolves the genus two finite blocking problem. Given a translation surface $(X, \omega)$ and two points $p$ and $q$ (that are not necessarily distinct and that may coincide with zeros of $\omega$) the finite blocking problem asks whether there is a finite collection of points $B$ on $(X, \omega) - \{p, q \}$ so that all straight lines between $p$ and $q$ must pass through $B$. The analogous problem on a polygon $T$ is whether given two point $p$ and $q$ there is a finite collection of points $B$ in $T - \{p, q\}$ so that all billiard shots from $p$ to $q$ pass through $B$. By unfolding the table $T$ to an abelian differential $(X, \omega)$ and letting $P$ and $Q$ be the preimages of $p$ and $q$ respectively, the finite blocking problem in polygonal biliards reduces to the question of whether all points in $P$ are finitely blocked from all points in $Q$ on $(X, \omega)$.

Leli\`evre, Monteil, and Weiss~\cite{LMW} showed that classifying invariant measures in $\mathcal{H}(2,0^2)$ and $\mathcal{H}(1,1,0^2)$ is related to solving the finite blocking problem. Such a classification is equivalent to Corollary~\ref{C1} and yields the following:

\begin{theorem}\label{T3}
A singular point on a primitive genus two translation surface is finitely blocked from no other point and a nonsingular point is only finitely blocked from its image under the hyperelliptic involution (and the blocking set is a subset of the collection of Weierstrass points).
\end{theorem}

Note that Leli\`evre, Monteil, and Weiss~\cite[Theorem 2]{LMW} implies that for any point $p$ on a primitive genus two translation surface, $p$ is finitely blocked from only finitely many points. Theorem~\ref{T3} determines exactly what that finite set is, i.e. empty if $p$ is a zero and a singleton containing the image of $p$ under the hyperelliptic involution otherwise. The strategy of this proof, which appears in Section~\ref{S:FBP2}, is to use the ideas of Leli\`evre, Monteil, and Weiss~\cite{LMW} alluded to above along with constraints on the geometry of eigenforms coming from McMullen's prototype surfaces in~\cite{McM:spin} and work of Leli\`evre and Weiss~\cite{LW} on convex representations of eigenforms.

To apply Theorem~\ref{T3} to polygons, we devise the following criterion:

\begin{theorem}\label{T:criterion}
A holomorphic $k$-differential on the sphere canonically unfolds to a hyperelliptic abelian differential only if one of the following two conditions holds:
\begin{enumerate}
\item All but at most two of the singularities have cone angle an integer multiple of $\frac{\pi}{2}$
\item The cone angles of the singularities, listed with multiplicity, are either $\left( a,a,b,b \right) 2\pi$ or $\left( a,a,b \right) 2\pi$ where $a$ and $b$ are rational.
\end{enumerate}
\end{theorem}

\noindent Theorem~\ref{T:criterion} is proved in Section~\ref{S:HyperBilliards} and necessary and sufficient conditions for the converse to hold are also established.



\begin{figure}[h!]
\begin{tabular}{m{3cm}m{1.5cm}m{6cm}m{2.5cm}}
\toprule
Polygon, $\theta =$ & $n=$ & Blocking & $\M(\theta)$ \\
\midrule
$\left( \frac{3}{2}, \left( \frac{1}{2} \right)^5 \right)\pi$ & $-$ & 	\begin{tikzpicture}
		\draw (0,0) -- (0,2) -- (1,2) -- (1,1) -- (2,1) -- (2,0) -- (0,0);
		\draw[dotted] (0,1) -- (1,1) -- (1,0);
		\node at (-.25, .5) {$y_1$}; \node at (-.25, 1.5) {$y_2$};
		\node at (.5, -.25) {$x_1$}; \node at (1.5, -.25) {$x_2$};
		\draw[black] (0,0) circle[radius=3pt];
		\draw[black] (0,2) circle[radius=3pt];
		\draw[black] (1,2) circle[radius=3pt];
		\draw[black] (2,1) circle[radius=3pt];
		\draw[black] (2,0) circle[radius=3pt];
	\end{tikzpicture} & $\strata(2)$ \\ 
$\left( \frac{1}{n}, \frac{n-1}{n}, \frac{1}{2}, \frac{1}{2} \right)\pi$ & $3,4$ & \begin{tikzpicture}
		\draw (0,0) -- (0, 1.7) -- (-1,1.7) -- (-2.5, 0) -- (0,0);
		\draw[dotted] (-1,1.7) -- (-1,0);
		\node at (-.5, -.25) {$x_2$}; \node at (-1.6, -.25) {$x_1$};
		\draw[black] (-2.5,0) circle[radius=3pt];
	\end{tikzpicture} \hspace{1cm} \begin{tikzpicture}
		\draw (0,0) -- (0, 1.7) -- (.5, 2.2) -- (2,0) -- (0,0);
		\draw[dotted] (0, 1.7) -- (.85, 1.7) -- (.85,0);
		\node at (.4, -.25) {$x_1$}; \node at (1.4, .-.25) {$x_2$};
		\draw[black] (2,0) circle[radius=3pt];
	\end{tikzpicture} & $\strata(1,1)$, $\strata(2)$ \\
$\left( \frac{1}{n}, \frac{1}{n}, \frac{n-1}{n}, \frac{n-1}{n} \right)\pi$ & $3$ & \begin{tikzpicture}
		\draw (0,0) -- ( .5, .85) -- (1.5, .85) -- (2,0) -- (0,0);
		\draw[dotted] (.5, .85) -- (.5, 0);
		\draw[dotted] (1.5, .85) -- (1.5, 0);
		\draw[black, fill] (1,.85) circle[radius=1.5pt];
		\draw[black, fill] (1,0) circle[radius=1.5pt];
		\node at (1, -.25) {$x_2$}; \node at (.25, -.25) {$x_1$};
		\draw[black] (0,0) circle[radius=3pt];
		\draw[black] (2,0) circle[radius=3pt];
	\end{tikzpicture} \hspace{1cm}  \begin{tikzpicture}
		\draw (0,0) -- ( .5, .85) -- (1.5, .85) -- (1,0) -- (0,0);
		\draw[dotted] (.5, .85) -- (.5, 0);
		\draw[dotted] (1, .85) -- (1, 0);
		\draw[black, fill] (.75,.425) circle[radius=1.5pt];
		\node at (.85, -.25) {$x_2$}; \node at (.25, -.25) {$x_1$};
		\draw[black] (0,0) circle[radius=3pt];
		\draw[black] (1.5,.85) circle[radius=3pt];
	\end{tikzpicture} & $\strata(1,1)$ \\
$\left( \frac{1}{2n}, \frac{n-1}{2n}, \frac{1}{2} \right)\pi$ & $4,5$ & \begin{tikzpicture}
		\draw (0,0) -- (2.5, 1) -- (2.5,0) -- (0,0);
		\draw[black] (0,0) circle[radius=3pt];
		\end{tikzpicture}  & Regular $2n$-gon locus \\
$\left( \frac{2}{2n}, \frac{n-2}{2n}, \frac{1}{2} \right)\pi$ & $5$ & \begin{tikzpicture}
		\draw (0,0) -- (2.5, 1) -- (2.5,0) -- (0,0);
		\end{tikzpicture} & Double pentagon locus  \\
$\left( \frac{1}{n}, \frac{1}{n}, \frac{n-2}{n} \right)\pi$ & $5,6$ & \begin{tikzpicture}
		\draw (0,0) -- (2.5, 1) -- (5,0) -- (0,0);
		\draw[black] (0,0) circle[radius=3pt];
		\draw[black] (5,0) circle[radius=3pt];
		\draw[black, fill] (2.5,0) circle[radius=3pt];
		\end{tikzpicture} & Double regular $n$-gon locus \\
$\left( \frac{2}{n}, \frac{2}{n}, \frac{n-4}{n} \right)\pi$ & $5$ & \begin{tikzpicture}
		\draw (0,0) -- (2.5, 1) -- (5,0) -- (0,0);
		\draw[black] (2.5,1) circle[radius=3pt];
		\draw[black, fill] (2.5,0) circle[radius=3pt];
		\end{tikzpicture} & Decagon locus \\
\bottomrule
\end{tabular}
\caption{Polygons that unfold to genus two translation surfaces}
\label{F:table}
\end{figure}

Given $\theta$, $\M(\theta)$ is the $\GL(2, \R)$ orbit closure of the unfolding of the generic polygon with rational angles $\theta$. Theorem~\ref{T:criterion} is used to construct a complete list of the billiard tables that unfold to genus two translation surfaces, and, along with Theorem~\ref{T3}, produces a solution to the finite blocking problem for these polygons (see Theorem~\ref{T:final}, which is proved in Section~\ref{S:Genus2}).



\begin{theorem}\label{T:final}
The complete list of rational $k$-gons that unfold to genus two translation surfaces is given in Figure~\ref{F:table}.
\begin{enumerate}
\item For $k = 6$, the surface is a torus cover if and only if $\frac{x_1}{x_2}$ and $\frac{y_1}{y_2}$ are rational.
\item For $k = 4$, the surface is a torus cover if and only if $\frac{x_1}{x_2}$ is rational.
\item For $k = 3$, only the $\left( \frac16, \frac16, \frac23 \right) \pi$ triangle unfolds to a torus cover.
\end{enumerate}
The solution to the finite blocking problem on the $k$-gons that do not unfold to torus covers is the following
\begin{enumerate}
\item For $k = 6$, the circled vertices are blocked from themselves.
\item For $k \ne 6$, if the figure has two circled vertices they are finitely blocked from each other by the solid points (which are midpoints of the faces or edges they are drawn on).
\item For $k \ne 6$, if the figure has one circled vertex it is blocked from itself except for $\theta = \left( \frac13, \frac23, \frac12, \frac12, \frac12 \right) \pi$, which has no pairs of finitely blocked points. The solid points form a blocking set.
\item In all other cases, there are no pairs of finitely blocked points. 
\end{enumerate}
\end{theorem}

Notice that for rational polygons that unfold to torus covers any two points are finitely blocked from each other. Therefore, Theorem~\ref{T:final} solves the finite blocking problem for rational polygons that unfold to genus two translation surfaces.

\bold{Acknowledgements.} The author thanks Alex Eskin, Alex Wright, and David Aulicino for extensive, detailed, and helpful conversations throughout the process of writing this paper. This material is based upon work supported by the National Science Foundation Graduate Research Fellowship Program under Grant No. DGE-1144082.

%
%

\section{Proof of Theorem~\ref{T2} - Periodic Points in Genus Two}\label{S:PT2}

Theorem~\ref{T2} has been established by M\"oller~\cite{M2} in the case of Teichm\"uller curves in genus two and by the author~\cite{Apisa-mp1} for strata of genus two abelian differentials. It remains to classify periodic points on nonarithmetic eigenform loci in $\mathcal{H}(1,1)$.

\begin{ass}\label{A:M-nonarithmetic}
$\M$ is a nonarithmetic eigenform locus in $\mathcal{H}(1,1)$.
\end{ass}

In every nonarithmetic eigenform locus $\mathcal{H}(1,1)$ there is a translation surface with three horizontal cylinders (this follows from Wright~\cite[Theorem 1.10]{Wcyl} by letting $(X, \omega)$ be a translation surface in $\M$ where the twist space coincides with the cylinder preserving space in the horizontal direction; it is also implicit in the work of McMullen~\cite[Section 7]{Mc5}). The horizontal cylinders will be glued together as in Figure~\ref{F:starter}. The specific horizontally and vertically periodic translation surface in Figure~\ref{F:starter} with the lengths of the horizontal and vertical saddle connections labelled as in the figure will be called the golden tetromino. We will let $\phi$ denote the golden ratio and $\zeta$ a vertical saddle connection contained in the topmost cylinder.

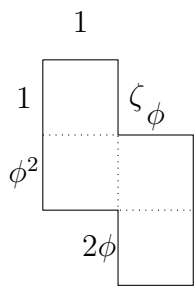
\begin{figure}[h]
\centering
	\begin{tikzpicture}
		\draw (0,1) -- (1,1) -- (1,0) -- (2,0) -- (2,-2) -- (1,-2) -- (1,-1) -- (0,-1) -- (0,1);
		\draw[dotted] (0,0) -- (1,0) -- (1,-1) -- (2,-1);
		\node at (1.25, .5) {$\zeta$}; 
		\node at (-.25, .5) {$1$};
		\node at (-.25, -.5) {$\phi^2$};
		\node at (.75, -1.5) {$2\phi$};
		\node at (.5, 1.5) {$1$};
		\node at (1.5, .25) {$\phi$};
	\end{tikzpicture}
\caption{The golden tetromino. Opposite sides identified.}
\label{F:starter}
\end{figure}

\begin{ass}\label{A:M-proto}
Fix a horizontally periodic translation surface $(X, \omega)$ with dense $\GL(2, \R)$ orbit in $\M$. Suppose that the horizontal cylinders are connected as in Figure~\ref{F:starter}.  Suppose too that $p$ is a periodic point that is not a Weierstrass point on $(X, \omega)$. Let $\M'$ be the orbit closure of $(X, \omega; p)$ in $\mathcal{H}(1,1,0)$. 
\end{ass}

Our goal is to show that $\M$ is the golden eigenform locus and that $p$ is one of the two golden points (which will be defined shortly). 

Label the horizontal cylinders of $(X, \omega)$ from top to bottom as $C_1$, $C_2$, and $C_3$. The relative deformation that only alters the imaginary parts of periods is $\rho  = -i \left( \gamma_1^* - \gamma_2^* + \gamma_3^* \right)$ where $\gamma_i$ is the core curve of $C_i$ traveling left to right. Let $t \rho \cdot (X, \omega)$ be the time $t$ flow along $\rho$ applied to $(X, \omega)$ (see Figure~\ref{fig:deformed-cyl}). This will be called the $\rho$-flow or the $\rho$-orbit of $(X, \omega)$. The portion of the $\rho$-flow when $t$ is positive will be called the forward $\rho$ flow and, when $t$ is negative, the backward (or reverse) $\rho$ flow. 

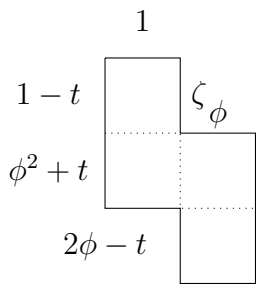
\begin{figure}[h]
\centering
	\begin{tikzpicture}
		\draw (0,1) -- (1,1) -- (1,0) -- (2,0) -- (2,-2) -- (1,-2) -- (1,-1) -- (0,-1) -- (0,1);
		\draw[dotted] (0,0) -- (1,0) -- (1,-1) -- (2,-1);
		\node at (1.25, .5) {$\zeta$}; 
		\node at (-.75, .5) {$1-t$};
		\node at (-.75, -.5) {$\phi^2+t$};
		\node at (0, -1.5) {$2\phi-t$};
		\node at (.5, 1.5) {$1$};
		\node at (1.5, .25) {$\phi$};

	\end{tikzpicture}
\caption{$t\rho \cdot (X, \omega)$ in the golden eigenform locus}
\label{fig:deformed-cyl}
\end{figure}

\begin{ass}\label{A:M-proto2}
Suppose without loss of generality, perhaps after rotating by $\pi$, that the topmost cylinder in $(X, \omega)$ is shorter than the bottomost. Suppose too that $(X, \omega)$ has been sheared so that the topmost cylinder contains a vertical saddle connection $\zeta$. In the case that $\M$ is the golden eigenform locus we will make the additional assumption that $(X, \omega)$ is some surface on the $\rho$ orbit of the golden tetromino. 
\end{ass}

Since the tangent spaces of $\M'$ and $\M$ are isomorphic under the forgetful map we will let, with some abuse of notation, $\rho$ also denote the lift of $\rho$ to $\M'$. 

\begin{prop}\label{P:boundary}
Let $(Y, \eta; q)$ be any genus two translation surface in the Mirzakhani-Wright partial compactification of $\M'$. The point $q$ coincides with a Weierstrass point or zero of $\eta$ if and only if $(Y, \eta)$ lies on the boundary of $\M$ or is contained in a Teichm\"uller curve in $\M$.
\end{prop}
\begin{proof}
By Mirzakhani-Wright~\cite[Theorem 1.1]{MirWri}, if two periodic points (or a periodic point and a zero) coincide they do so along a proper affine invariant submanifold in $\M$. Since $\M$ is three-dimensional, this must be either a Teichm\"uller curve in $\M$ or in its boundary.

For the reverse direction, if $(Y, \eta)$ is a genus two translation surface lying in a Teichm\"uller curve in the Mirzakhani-Wright partial compactification of $\M$, then either it is an eigenform locus in $\mathcal{H}(2)$ or it is the decagon locus in $\mathcal{H}(1,1)$. M\"oller~\cite{M2} showed that the only periodic points in these loci are Weierstrass points. By Mirzakhani-Wright~\cite[Theorem 1.1]{MirWri}, the restriction of $\M'$ to any Teichm\"uller curves it contains will be a periodic point or zero on the translation surfaces in those Teichm\"uller curves.
\end{proof}

\begin{cor}
If $\M'$ contains a translation surface $(Y, \eta; q)$ where $q$ is a Weierstrass point, then $\M$ is the golden eigenform locus.
\end{cor}
\begin{proof}
By Proposition~\ref{P:boundary}, $(Y, \eta)$ must lie on a Teichm\"uller curve contained in $\M$. The only eigenform locus containing a Teichm\"uller curve is the golden eigenform locus, which contains the decagon locus. 
\end{proof}

The strategy of the proof will be to show that it is always possible to find a translation surface in $\M'$ where the periodic point coincides with a Weierstrass point. This will imply that $\M$ is the golden eigenform locus.


\begin{lemma}\label{L:WP}
The core curve of every cylinder in $(X, \omega)$ contains exactly two Weierstrass points.
\end{lemma}
\begin{proof}
Since $(X, \omega)$ belongs to a hyperelliptic component, the core curve of every cylinder is fixed by the hyperelliptic involution, which is an orientation reversing isometry of the core curve. The result follows since every orientation reversing isometry of a circle has exactly two fixed points.
\end{proof}


\begin{prop}\label{P:t=0}
Let $(X, \omega)$ be the golden tetromino. Let $(X_t, \omega_t)$ be the translation surface obtained by flowing  along $t \rho$ for $t$ some value in $(-\phi^2, 1)$. Then $(X_t, \omega_t)$ belongs to the decagon locus if and only if $t = 0$.
\end{prop}
\begin{proof}
The decagon locus has only two cusps, one which corresponds to a three-cylinder direction with ratio of moduli $[1:1:2]$ and a second corresponding to a two-cylinder direction with identical moduli. Since the three horizontal cylinders persists for all $t$ in $(-\phi^2, 1)$, if $(X_t, \omega_t)$ lies in the decagon locus then their ratio of moduli must be $[1:1:2]$. At time $t$ the modulus of the topmost cylinder is $1-t$ and the modulus of the bottommost is $2 - \frac{t}{\phi}$. Since $1 - t < 2 - \frac{t}{\phi}$ for all $t > -\phi^2$, it follows that $(X_t, \omega_t)$ belongs to the decagon locus only if $2(1-t) = 2 - \frac{t}{\phi}$. The only solution to this equation is $t = 0$. 
\end{proof}


\begin{lemma}\label{L:boundary}
The marked point does not lie on the boundary of a horizontal cylinder for its entire $\rho$ orbit. 
\end{lemma}
\begin{proof}
Suppose not to a contradiction. By Assumption~\ref{A:M-proto2}, $(X, \omega)$ has a vertical saddle connection $\zeta$ whose period tends to zero under forward $\rho$-flow. Let $(X', \omega'; p')$ be the marked translation surface in the boundary of $\M'$ that forms when the period of $\zeta$ vanishes under the forward $\rho$ orbit of $(X, \omega; p)$. By Proposition~\ref{P:boundary}, $p'$ is a Weierstrass point. Since the Weierstrass points move continuously, there is a well-defined Weierstrass point $q$ on $(X, \omega)$ so that $p$ and $q$ coincide on $(X', \omega')$. 

Cylinders $C_2$ and $C_3$ persist on $(X', \omega')$. By Lemma~\ref{L:WP}, there are two Weierstrass points in their interiors on $(X', \omega')$.  The other two Weierstrass points on $(X', \omega')$ occur at the unique zero of $\omega'$ and at the midpoint of the saddle connection that connected $C_1$ and $C_2$ on $(X, \omega)$. 

We have assumed that the marked point lies on the boundary of a horizontal cylinder under forward $\rho$ flow. Therefore, it cannot coincide with a zero on $(X', \omega')$ since it would have to be a zero on $(X, \omega)$. It follows that $p$ lies at the midpoint of one of the two horizontal saddle connections on $(X, \omega)$ that connect $C_1$ and $C_2$ (see Figure~\ref{fig:noboundary1}).



\begin{figure}[h]
\centering
	\begin{tikzpicture}
		\draw (0,1) -- (1,1) -- (1,0) -- (2,0) -- (2,-2) -- (1,-2) -- (1,-1) -- (0,-1) -- (0,1);
		\draw[dotted] (0,0) -- (1,0) -- (1,-1) -- (2,-1);
		\draw[black, fill] (.5,0) circle[radius=1.5pt]; \node at (.25,0) {$p$};
		\draw[black, fill] (.5,.5) circle[radius=1.5pt]; \node at (.25,.5) {$q$};
	\end{tikzpicture}
\caption{Starting configuration}
\label{fig:noboundary1}
\end{figure}
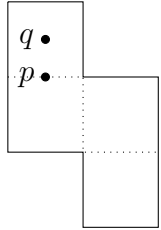

Let $q$ be the Weierstrass point on $(X, \omega)$ that merges with $p$ on $(X', \omega')$. Let $\gamma$ be the vertical line that travels up from $p$ to $q$. For any sufficiently small perturbation of $(X, \omega)$ that keeps $\zeta$ vertical, Proposition~\ref{P:boundary} implies that using forward $\rho$ flow to shrink $\zeta$ causes $p$ and $q$ to coincide exactly when $\zeta$ vanishes. Since the period of $\gamma$ vanishes exactly when the period of $\zeta$ vanishes, $\M'$ locally satisfies the equation $\gamma = \frac{\zeta}{2}$. Now, horizontally shear the surface (as in Figure~\ref{fig:noboundary2}) so that $p$ and a zero are on a vertical line contained in $C_1$. 

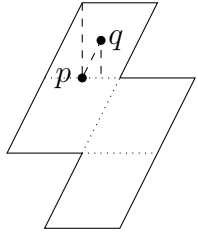
\begin{figure}[h]
\centering
	\begin{tikzpicture}
		\draw (.5,1) -- (1.5,1) -- (1,0) -- (2,0) -- (1,-2) -- (0,-2) -- (.5,-1) -- (-.5,-1) -- (.5,1);
		\draw[dotted] (0,0) -- (1,0) -- (.5,-1) -- (1.5,-1);
		\draw[black, fill] (.5,0) circle[radius=1.5pt]; \node at (.25,0) {$p$};
		\draw[black, fill] (.75,.5) circle[radius=1.5pt]; \node at (.95,.5) {$q$};
		\draw[dashed] (.5, 1) -- (.5, 0) -- (.75, .5) -- (.75, 0);
	\end{tikzpicture}
\caption{Sheared configuration}
\label{fig:noboundary2}
\end{figure}

Flow forward under $\rho$ until the height of $C_1$ becomes zero. Call the new translation surface $(Y, \eta)$. Since no saddle connection has vanished, $(Y, \eta)$ does not belong to the boundary of $\M$. Moreover, the equation $\gamma = \frac{\zeta}{2}$ implies that $p$ and $z$ coincide on $(Y, \eta)$. By Proposition~\ref{P:boundary}, $(Y, \eta)$ must belong to a Teichm\"uller curve and hence it belongs to the decagon locus and so $\M$ is the golden eigenform locus.

When $\M$ is the golden eigenform locus, Assumption~\ref{A:M-proto2} implies that $(X, \omega)$ is a surface on the $\rho$ orbit of the golden tetromino. Since $p$ is fixed under $\rho$ flow at the midpoint of one of the two saddle connections that connect $C_1$ and $C_2$ it remains that point on the golden tetromino. However, this point is not a Weierstrass point on the golden tetromino and hence not a periodic point by M\"oller~\cite{M2}. This is a contradiction. 
\end{proof}

In light of Lemma~\ref{L:boundary}, we make the following assumption.

\begin{ass}\label{A:M-interior}
Suppose without loss of generality, possibly after flowing along $\rho$ for an arbitrarily small amount of time, that $p$ belongs to the interior of a horizontal cylinder on $(X, \omega)$.
\end{ass}


\begin{lemma}\label{L:singlecylinder}
The marked point remains in the interior of a single cylinder for its entire $\rho$ orbit.
\end{lemma}
\begin{proof}
Suppose it does not, i.e. there is a real number $t$ so that for all $s$ in $[0, t]$, the three horizontal cylinders persist on $s\rho \cdot (X, \omega)$ and so that at time $t$ the marked point reaches the boundary of one cylinder, say $C$. The proof will proceed in three steps. First we will show that $\M$ must be the golden eigenform locus. Second we will show that the marked point only reaches the boundary of a cylinder on the golden tetromino $(X_0, \omega_0)$. We show too that flowing backwards along $\rho$ from $(X_0, \omega_0)$ the marked point must be contained in the middle horizontal cylinder. Finally, we show that even this cannot occur.

\noindent \textbf{Step 1: $\M$ is the golden eigenform locus}

Let $z$ be a zero on the boundary of the cylinder that $p$ reaches at time $t$ (see Figure~\ref{fig:samecylinder1}).

\begin{figure}[h]
\centering
	\begin{tikzpicture}
		\draw (0,1) -- (1,1) -- (1,0) -- (2,0) -- (2,-2) -- (1,-2) -- (1,-1) -- (0,-1) -- (0,1);
		\draw[dotted] (0,0) -- (1,0) -- (1,-1) -- (2,-1);
		\draw[black, fill] (0,0) circle[radius=1.5pt]; \node at (.25,.0) {$z$};
		\draw[black, fill] (.25,-.5) circle[radius=1.5pt]; \node at (.5,-.5) {$p$};
	\end{tikzpicture}
\caption{Starting Configuration}
\label{fig:samecylinder1}
\end{figure}
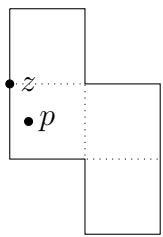

\noindent Shear so that $C$ contains a vertical line joining $p$ and $z$, as in Figure~\ref{fig:samecylinder2}.

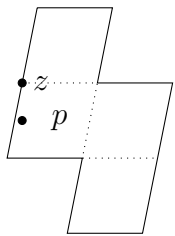
\begin{figure}[h]
\centering
	\begin{tikzpicture}
		\draw (.2,1) -- (1.2,1) -- (1,0) -- (2,0) -- (1.6,-2) -- (.6,-2) -- (.8,-1) -- (-.2,-1) -- (.2,1);
		\draw[dotted] (0,0) -- (1,0) -- (.8,-1) -- (1.8,-1);
		\draw[black, fill] (0,0) circle[radius=1.5pt]; \node at (.25,0) {$z$};
		\draw[black, fill] (0,-.5) circle[radius=1.5pt]; \node at (.5,-.5) {$p$};
	\end{tikzpicture}
\caption{Sheared configuration}
\label{fig:samecylinder2}
\end{figure}

Flow along $\rho$ so that the marked point $p$ moves straight into zero at time $s$. Since all three horizontal cylinders persist and the periodic point has been moved into a zero, $\M$ must contain a Teichm\"uller curve by Proposition~\ref{P:boundary}. Therefore, $\M$ is the golden eigenform locus. Recall that we have assumed in this case that $(X, \omega)$ is a surface given by applying $\rho$ for some amount of time to the golden tetromino. Assume specifically that $(X, \omega) = -\rho \cdot (X_0, \omega_0)$ where $(X_0, \omega_0)$ is the golden tetromino.

\noindent \textbf{Step 2: $p$ belongs to the middle horizontal cylinder on $(X, \omega)$ }

 By Proposition~\ref{P:t=0}, the $\rho$ flow from the golden tetromino lies in the decagon locus only at time zero. Therefore, the previous argument implies that the only surface on the $\rho$ flow of the golden tetromino on which $p$ meets the boundary of a horizontal cylinder is on the golden tetromino itself. Suppose to a contradiction that this occurs.

Since the only periodic points on the golden tetromino are Weierstrass points, it follows that $p$ must collide with a zero on the golden tetromino. In particular, this implies that it lies on a vertical separatrix. Suppose to a contradiction that $p$ belongs to cylinder $C$ that is not the middle cylinder $C_2$. Let $\xi$ be the dashed saddle connection depicted in Figure~\ref{F:zeta}


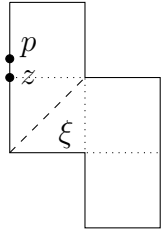
\begin{figure}[h]
\centering
	\begin{tikzpicture}
		\draw (0,1) -- (1,1) -- (1,0) -- (2,0) -- (2,-2) -- (1,-2) -- (1,-1) -- (0,-1) -- (0,1);
		\draw[dotted] (0,0) -- (1,0) -- (1,-1) -- (2,-1);
		\draw[black, fill] (0,0) circle[radius=1.5pt]; \node at (.25,.0) {$z$};
		\draw[black, fill] (0,.25) circle[radius=1.5pt]; \node at (.25,.4) {$p$};
		\draw[dashed] (0,-1) -- (1,0); \node at (.75, -.75) {$\xi$};
	\end{tikzpicture}
\caption{A saddle connection to collapse}
\label{F:zeta}
\end{figure}

Shear so that the saddle connection $\xi$ is vertical and then flow backwards along $\rho$ to collapse $\xi$ - forming a boundary translation surface $(Y, \eta)$. The boundary translation surface $(Y, \eta)$ belongs to a eigenform locus in $\mathcal{H}(2)$. By Proposition~\ref{P:boundary} the marked point $p$ must coincide with a Weierstrass point $w$ on $(Y, \eta)$. By abuse of notation, we will let $w$ be the Weierstrass point on $(X, \omega)$ that $p$ collides with under this degeneration. 

We will argue that $w$ is contained in the interior of the cylinder $C$ that contains $p$. This follows immediately from the observation that on $(Y, \eta)$ the periodic point $p$ must lie exactly on the central horizontal curve in $C$. If this did not happen, then under reverse $\rho$ flow starting at the golden tetromino, $p$ would travel from one boundary curve of $C$ to another. However, that would imply that $p$ at some point passed through a Weierstrass point. By Proposition~\ref{P:boundary} that can happen if and only if the underlying translation surface becomes a Veech surface, but by Proposition~\ref{P:t=0} that cannot happen. 

Let $\gamma$ be the straight line path from $p$ to $w$ that is contained in the cylinder $C$ and whose period goes to zero as the period of $\xi$ tends to zero. There is a constant $c$ so that $\gamma = c \xi$ holds on all surfaces in a neighborhood of $(X, \omega; p)$ in $\M'$. In particular, the slope of $\gamma$ coincides with the slope of $\xi$. However, passing to the golden tetromino under forward $\rho$ flow from $(X, \omega)$ we have that $\gamma$ has rational slope (since $p$ is contained in $C_1$ or $C_3$) and $\xi$ has irrational slope. This contradicts the fact that $\gamma$ are $\xi$ are parallel. Suppose therefore that $p$ lies in the middle horizontal cylinder on $(X, \omega)$. 

\noindent \textbf{Step 3: $p$ cannot belong to the middle horizontal cylinder on $(X, \omega)$ }

Without loss of generality suppose that $p$ lies in the upper half of the middle horizontal cylinder and that $p$ intersects the top boundary of that cylinder under forward $\rho$ flow to the golden tetromino. Shear the surface so that there is a Weierstrass point $w$ in $C_1$ that lies directly above $p$. Call the resulting surface $(Y, \eta)$ and let $\gamma$ be the vertical line from $p$ to $w$, see Figure~\ref{F:gamma-sc}.


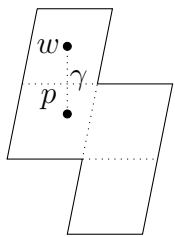
\begin{figure}[h]
\centering
	\begin{tikzpicture}
		\draw (.2,1) -- (1.2,1) -- (1,0) -- (2,0) -- (1.6,-2) -- (.6,-2) -- (.8,-1) -- (-.2,-1) -- (.2,1);
		\draw[dotted] (0,0) -- (1,0) -- (.8,-1) -- (1.8,-1);
		\draw[black, fill] (.6,-.4) circle[radius=1.5pt]; \node at (.35,-.2) {$p$};
		\draw[black, fill] (.6,.5) circle[radius=1.5pt]; \node at (.35,.5) {$w$};
		\draw[dotted] (.6, -.4) -- (.6, .5); \node at (.75, .05) {$\gamma$};
	\end{tikzpicture}
\caption{The saddle connection on $(Y, \eta)$}
\label{F:gamma-sc}
\end{figure}



Let $(Y_t, \eta_t)$ be the surface that is the time $t$ flow along $\rho$ starting at the $(Y, \eta)$. Recall that we chose $(X, \omega)$ so that at time $1$ under the forward $\rho$ flow it would become the golden tetromino. The computation in Proposition~\ref{P:t=0} shows that the only time $t$ in $[0, 2]$ for which two horizontal cylinders in $(Y_t, \eta_t)$ have identical moduli is time $1$. As mentioned in that proof, any three-cylinder direction in the decagon locus contains two cylinders of identical moduli. Therefore, $(Y_t, \eta_t)$ belongs to the decagon locus only at time $1$. 

Notice that $p$ belongs to the interior of $C_1$ and is moving upward for all time $t > 1$ (by moving upward we mean that the height of the point relative to the height of $C_1$ is increasing). Since $w$ lies directly above $p$ and lies on the core curve of $C_1$ we see that $p$ and $w$ must collide before $C_1$ fully collapses. By Proposition~\ref{P:boundary} when $p$ and $w$ collide the underlying translation surface must belong to the decagon locus. However, we have already argued that $(Y_t, \eta_t)$ does not belong to the decagon locus for $t > 1$, which is a contradiction.  


\end{proof}

Recall our standing assumption (Assumption~\ref{A:M-proto2}) that $(X, \omega)$ is a surface with a vertical saddle connection $\zeta$ whose length tends to zero as we flow along $\rho$. 

\begin{lemma}\label{L:cylinder1} 
The marked point is not contained in the topmost cylinder.
\end{lemma}
\begin{proof}
Let $(Y, \eta)$ be the translation surface on the boundary of $\M$ that is formed when the period of $\zeta$ vanishes under forward $\rho$ flow. By Proposition~\ref{P:boundary} the periodic point $p$ must coincide with a Weierstrass point on $(Y, \eta)$. In particular, this means that it must lie on a vertical separatrix contained in $C_1$ joining $p$ to a Weierstrass point. Shear the surface to perform one complete Dehn twist in $C_1$. If there is still a vertical separatrix contained in $C_1$ joining $p$ to a Weierstrass point then $p$ must itself be a Weierstrass point or a zero of $\omega$. Since $p$ is neither, it does not lie directly above or below a Weierstrass point after the shear is performed. Therefore, flowing under forward $\rho$ flow from the sheared surface collapses a vertical saddle connection contained in $C_1$ and passes to a Teichm\"uller curve in the boundary of $\M$, but it does not cause the periodic point to coincide with a Weierstrass point on the boundary translation surface. This contradicts Proposition~\ref{P:boundary}. 
\end{proof}

The previous two lemmas are sufficient to classify the periodic points in the golden eigenform locus. 

\begin{prop}
The only periodic points in the golden eigenform locus are the golden points.
\end{prop}
\begin{proof}
Let $(X_0, \omega_0)$ be the golden tetromino and let $(X_t, \omega_t) := t \rho \cdot (X_0, \omega_0)$ for $t$ in $(-\phi^2, 1)$. By Lemma~\ref{L:singlecylinder}, $p$ remains in the interior of a single horizontal cylinder for all $t$. By Lemma~\ref{L:cylinder1} this cylinder is not the topmost cylinder. By M\"oller~\cite{M2}, on $(X_0, \omega_0)$ the marked point coincides with a Weierstrass point in the interior of $C_2$ or $C_3$. On $(X_1, \omega_1)$ the saddle connection $\zeta$ has vanished and so $p$ must coincide with a Weierstrass point or zero by Proposition~\ref{P:boundary}. Therefore, on the golden tetromino, the periodic point $p$ is one of two Weierstrass points shown as black dots in Figure~\ref{fig:cyl}


\begin{figure}[h]
\centering
	\begin{tikzpicture}
		\draw (0,1) -- (1,1) -- (1,0) -- (2,0) -- (2,-2) -- (1,-2) -- (1,-1) -- (0,-1) -- (0,1);
		\draw[dotted] (0,0) -- (1,0) -- (1,-1) -- (2,-1);
		\node at (1.25, .5) {$\zeta$}; 
		\node at (-.25, .5) {$1$};
		\node at (-.25, -.5) {$\phi^2$};
		\node at (.5, -1.5) {$2\phi$};
		\node at (.5, 1.5) {$1$};
		\node at (1.5, .25) {$\phi$};
		
		\draw[black, fill] (.5,-.5) circle[radius=1.5pt]; \node at (.5, -.2) {$w_1$};
		\draw[black, fill] (1,-1.5) circle[radius=1.5pt];\node at (1.4, -1.5) {$w_2$}; 

	\end{tikzpicture}
\caption{Two possible marked points}
\label{fig:cyl}
\end{figure}
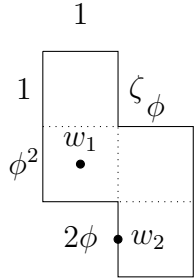

Suppose first that the marked point $p$ coincides with $w_1$ on the golden tetromino. Perturb slightly using $\rho$ so that the marked point now lies slightly above or below $w_1$. Shear the vertical cylinders so that one complete Dehn twist is performed about each vertical cylinder - this is possible since the vertical cylinders have identical moduli. The marked point now is contained in $C_1$ on the $\rho$-orbit of the golden tetromino, which contradicts Lemma~\ref{L:cylinder1}. 

Therefore, the marked point must coincide with $w_2$ on the golden tetromino. Since the periodic point coincides with a zero $z$ on the boundary of $C_3$ when $t = 1$, we see that if $\gamma$ is the straight line between $p$ and $z$, that $\M'$ satisfies the equation $\gamma = \phi \zeta$. By Kumar-Mukamel~\cite{MK16}, this equation defines the golden points.
\end{proof}

Since we established Theorem~\ref{T2} for the golden eigenform locus, we will make the following standing assumption for the remainder of the section:

\begin{ass}\label{A:golden}
$\M$ is not the golden eigenform locus 
\end{ass}

Our goal will be to derive a contradiction from this assumption and Assumption~\ref{A:M-proto}.

\begin{lemma}\label{L:half}
The marked point remains in either the top or bottom half of a single horizontal cylinder. 
\end{lemma}
\begin{proof}
Suppose not to a contradiction and suppose without loss of generality that $p$ begins in the top half of a cylinder a crosses to the bottom half under forward $\rho$ flow. Shear the surface so that $p$ lies above a Weierstrass point. Flowing by $\rho$ now causes the marked point $p$ to collide with a Weierstrass point without degenerating the surface. By Proposition~\ref{P:boundary}, $\M$ contains a Teichm\"uller curve. However, the golden eigenform locus is the only such eigenform locus and $\M$ is not that locus by Assumption~\ref{A:golden}.
\end{proof}



\begin{lemma}\label{L:middle-cylinder}
The marked point is not contained in the middle horizontal cylinder.
\end{lemma}
\begin{proof}
Suppose not to a contradiction. Assume without loss of generality (possibly using the hyperelliptic involution) that $p$ belongs to the top half of the middle cylinder. Let $B$ be the top boundary of the cylinder. 

\noindent \textbf{Case 1: $p$ moves towards $B$ under forward $\rho$ flow}


By assumption, the distance between $p$ and $B$ decreases under forward $\rho$ flow (and so it increases under backwards $\rho$ flow). Under backwards $\rho$ flow the height of the middle cylinder tends to zero. Therefore, shear the surface so that $p$ lies directly above a Weierstrass point contained in the middle cylinder and flow backwards along $\rho$. Since the distance from the Weierstrass point and $B$ tends to zero and the distance between $p$ and $B$ increases, $p$ must collide with a Weierstrass point before the surface degenerates. By Proposition~\ref{P:boundary}, $\M$ contains a Teichm\"uller curve and this contradicts Assumption~\ref{A:golden}.


\noindent \textbf{Case 2: $p$ moves towards the core curve under forward $\rho$ flow}


Shear the middle horizontal cylinder so that it contains a vertical saddle connection $\xi$ with the property that sending the period of $\xi$ to zero under reverse $\rho$ flow passes to a boundary translation surface in $\strata(2)$, see Figure~\ref{fig:final-stage}. By Proposition~\ref{P:boundary}, $p$ collides with a Weierstrass point on the boundary and therefore $p$ lies directly above a Weierstrass point contained in the middle horizontal cylinder. 

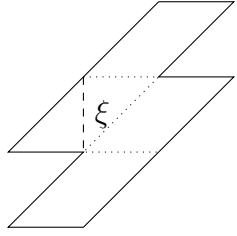
\begin{figure}[h]
\centering
	\begin{tikzpicture}
		\draw (1,1) -- (2,1) -- (1,0) -- (2,0) -- (0,-2) -- (-1,-2) -- (0,-1) -- (-1,-1) -- (1,1);
		\draw[dashed] (0,0) -- (0,-1); \node at (.25, -.5) {$\xi$};
		\draw[dotted] (0,0) -- (1,0) -- (0, -1) -- (1, -1);
	\end{tikzpicture}
\caption{Sheared configuration}
\label{fig:final-stage}
\end{figure}

Shear the surface so as to perform one complete Dehn twist in the middle horizontal cylinder. As before it is still necessarily the case that $p$ lies directly above a Weierstrass point contained in the middle horizontal cylinder. Since $p$ belongs to the interior of the middle horizontal cylinder, $p$ lies directly above a Weierstrass point both before and after the shear if and only if $p$ is a Weierstrass point, which is a contradiction. 


\end{proof}

\begin{proof}[Proof of Theorem~\ref{T2}:]
By Lemmas~\ref{L:cylinder1} and \ref{L:middle-cylinder} we may suppose that $p$ is contained in the interior of the bottom horizontal cylinder. By Lemma~\ref{L:half} we suppose without loss of generality that $p$ is contained in the top half of the bottom horizontal cylinder for the entire $\rho$ flow. Let $B$ be the top boundary of the bottom horizontal cylinder. Let $h_i$ denote the height of $C_i$ and let $\ell_i$ be the length of its core curve. Suppose without loss of generality, after shearing, that the middle horizontal cylinder begins in the configuration shown in Figure~\ref{fig:cyl2study1}.

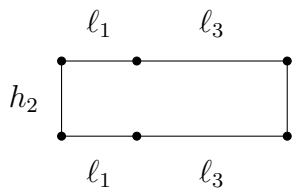
\begin{figure}[h]
\centering
	\begin{tikzpicture}
		\draw (0,0) -- (0,1) -- (3,1) -- (3,0) -- (0,0);
		\draw[black, fill] (0,0) circle[radius=1.5pt]; 
		\draw[black, fill] (0,1) circle[radius=1.5pt]; 
		\draw[black, fill] (1,1) circle[radius=1.5pt]; 
		\draw[black, fill] (1,0) circle[radius=1.5pt]; 
		\draw[black, fill] (3,1) circle[radius=1.5pt]; 
		\draw[black, fill] (3,0) circle[radius=1.5pt]; 
		\node at (.5, 1.5) {$\ell_1$};
		\node at (.5, -.5) {$\ell_1$};
		\node at (2, 1.5) {$\ell_3$};
		\node at (2, -.5) {$\ell_3$};
		\node at (-.5, .5) {$h_2$};
	\end{tikzpicture}
\caption{Starting configuration}
\label{fig:cyl2study1}
\end{figure}

\noindent \textbf{Case 1: $p$ moves towards $B$ under forward $\rho$ flow}

Suppose first that $\ell_1 < \ell_3$. After shearing by $\begin{pmatrix} 1 & \frac{\ell_1}{h_2} \\ 0 & 1 \end{pmatrix}$ we have the configuration shown in Figure~\ref{fig:cyl2study2}.

\begin{figure}[h]
\centering
	\begin{tikzpicture}
		\draw (0,0) -- (1,1) -- (4,1) -- (3,0) -- (0,0);
		\draw[black, fill] (0,0) circle[radius=1.5pt]; 
		\draw[black, fill] (4,1) circle[radius=1.5pt]; 
		\draw[black, fill] (1,1) circle[radius=1.5pt]; 
		\draw[black, fill] (1,0) circle[radius=1.5pt]; 
		\draw[black, fill] (2,1) circle[radius=1.5pt]; 
		\draw[black, fill] (3,0) circle[radius=1.5pt]; 
		\node at (1.5, 1.5) {$\ell_1$};
		\node at (.5, -.5) {$\ell_1$};
		\node at (3, 1.5) {$\ell_3$};
		\node at (2, -.5) {$\ell_3$};
		\node at (-.5, .5) {$h_2$};
	\end{tikzpicture}
\caption{First sheared configuration}
\label{fig:cyl2study2}
\end{figure}
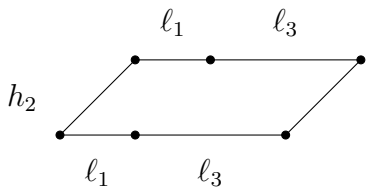

Similarly shearing by $\begin{pmatrix} 1 & \frac{\ell_3}{h_2} \\ 0 & 1 \end{pmatrix}$ we have the configuration shown in Figure~\ref{fig:cyl2study3}.

\begin{figure}[h]
\centering
	\begin{tikzpicture}
		\draw (0,0) -- (2,1) -- (5,1) -- (3,0) -- (0,0);
		\draw[black, fill] (0,0) circle[radius=1.5pt]; 
		\draw[black, fill] (2,1) circle[radius=1.5pt]; 
		\draw[black, fill] (3,1) circle[radius=1.5pt]; 
		\draw[black, fill] (5,1) circle[radius=1.5pt]; 
		\draw[black, fill] (1,0) circle[radius=1.5pt]; 
		\draw[black, fill] (3,0) circle[radius=1.5pt]; 
		\node at (2.5, 1.5) {$\ell_1$};
		\node at (.5, -.5) {$\ell_1$};
		\node at (3.5, 1.5) {$\ell_3$};
		\node at (2, -.5) {$\ell_3$};
		\node at (-.5, .5) {$h_2$};
	\end{tikzpicture}
\caption{Second sheared configuration}
\label{fig:cyl2study3}
\end{figure}
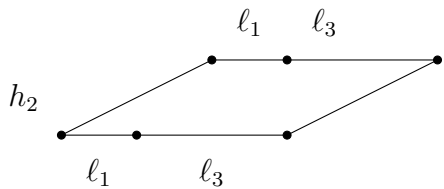


In these two sheared configurations there is exactly one vertical saddle connection contained in the middle horizontal cylinder and so collapsing it under backwards $\rho$ flow passes to a Veech surface in $\strata(2)$ on the boundary of $\M$. By Proposition~\ref{P:boundary}, on this boundary translation surface the periodic point coincides with a Weierstrass point. Therefore, in these two sheared configurations, $p$ lies directly above a Weierstrass point contained in $C_3$.

Let's begin in the first sheared configuration (Figure~\ref{fig:cyl2study2}). The marked point $p$ lies directly above a Weierstrass point $w$. Let $w'$ be the other Weierstrass point contained in the bottom horizontal cylinder. As we continue to shear the surface if at some point $p$ lies above $w'$, then $C_2$ must contain a vertical saddle connection (if not we can apply reverse $\rho$ flow to cause the periodic point to coincide with a Weierstrass point while remaining in $\M$; this contradicts Assumption~\ref{A:golden}). Conversely, if $C_2$ contains exactly one vertical saddle connection, then $p$ must lie above a Weierstrass point since collapsing $C_2$ using reverse $\rho$ flow passes to a Veech surface in $\mathcal{H}(2)$. These two observations imply that as we shear from the first sheared configuration to the second the periodic point goes from lying directly above $w$ to lying directly above $w'$ without lying over a Weierstrass point at any other point along the way. In other words, the matrix $\begin{pmatrix} 1 & \frac{\ell_3 - \ell_1}{h_2} \\ 0 & 1 \end{pmatrix}$ moves the marked point a horizontal distance of $\frac{\ell_3}{2}$ to the right of $w$. This implies that the marked point lies at height $\frac{h_2 \ell_3}{2 \left( \ell_3 - \ell_1 \right)}$ above $w$. 

However, beginning in the first sheared configuration and doing a complete Dehn twist also moves the periodic point by an integer multiple of $\frac{\ell_3}{2}$ horizontally with respect to $w$, which implies that $p$ lies at height $\frac{n h_2 \ell_3}{2 \left( \ell_3 + \ell_1 \right)}$ above $w$, where $n$ is an integer. These two equations for the height of $p$ above $w$ imply that $\frac{\ell_3 - \ell_1}{\ell_3 + \ell_1}$ is a rational number, which occurs if and only if $\frac{\ell_1}{\ell_3}$ is rational. However, this cannot be since this would imply that the original surface is arithmetic, contradicting Assumption~\ref{A:M-nonarithmetic}. The case of $\ell_3 < \ell_1$ is identical except the expression $\ell_3 - \ell_1$ is replaced with $\ell_1 - \ell_3$. The case of $\ell_1 = \ell_3$ cannot occur since this again would force the translation surface to be a torus cover.

\noindent \textbf{Case 2: $p$ moves towards the core curve under forward $\rho$ flow}

When $\zeta$ vanishes under forward $\rho$ flow, the translation surface becomes a nonarithmetic eigenform in $\mathcal{H}(2)$. By Proposition~\ref{P:boundary}, when this happens, the marked point must coincide with a Weierstrass point in the interior of $C_3$. 

Now shear $(X, \omega)$ so that $C_1$ contains no vertical saddle connection. Flow forward using $\rho$ until just after the height of $C_1$ has become zero. Call the resulting translation surface $(Y, \eta)$. It is shown in Figure~\ref{F:overcollapse-c1}. Opposite sides are identified unless otherwise specified. Solid dots in the bottom cylinder are Weierstrass points. 

\begin{figure}[h]
\centering
	\begin{tikzpicture}
		\draw (0,5) -- (1, 4) -- (3, 5) -- (4,5) -- (4, 0) -- (3,0) -- (3,2) -- (2,3) -- (0,2) -- (0,5);
		\draw[dotted] (0,4) -- (4,4);
		\draw[dotted] (0,3) -- (4,3);
		\draw[dotted] (3,2) -- (4,2);
		\draw[black, fill] (3,1) circle[radius=1.5pt];
		\draw[black, fill] (3.5,1) circle[radius=1.5pt];
		\draw[black] (3.2, .5) circle[radius=1.5pt];
		\node at (.5, 4.75) {$a$}; \node at (2, 4.75) {$b$};
		\node at (1, 2.25) {$b$}; \node at (2.5, 2.25) {$a$};
		\node at (3.5, .5) {$p$};
		\node at ( -.5, 4.5) {$C_2'$}; \node at ( -.5, 3.5) {$C_1'$}; \node at ( -.5, 1) {$C_3'$};
	\end{tikzpicture}
	\caption{Overcollapsing $C_1$}
	\label{F:overcollapse-c1}
\end{figure}
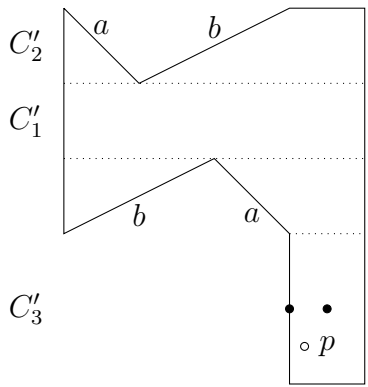

Relabel the horizontal cylinders in $(Y, \eta)$ as in Figure~\ref{F:overcollapse-c1} and, for $i \in \{1, 2, 3\}$, let $\gamma_i'$ denote the core curve of $C_i'$ oriented from left to right. Notice that continuing to flow along $\rho$ is the same as flowing along $\rho' := -i \left( \gamma_1' - \gamma_2' + \gamma_3' \right)$. After shearing so that the shorter cylinder of $\{ C_1', C_3'\}$ contains a vertical saddle connection we notice that $(Y, \eta)$ satisfies Assumptions~\ref{A:M-proto},~\ref{A:M-proto2}. and~\ref{A:M-interior}. If $C_3'$ is shorter than $C_1'$ then we contradict Lemma~\ref{L:cylinder1}. Otherwise, $p$ is contained in the bottom half of $C_3'$ and moves towards the bottom boundary as the height of $C_1'$ decreases. This is Case 1 of this proof; hence we have a contradiction as desired.

\end{proof}

%
%

\section{Proof of Theorem~\ref{T3} - The Finite Blocking Problem in Genus Two}~\label{S:FBP2}

The finite blocking problem asks whether given two points $p$ and $q$ on a translation surface $(X, \omega)$ there is a finite collection of points $B$ on $(X, \omega)$ so that all straight line paths between $p$ and $q$ contain a point in $B$? In this section we will solve the finite blocking problem for genus two translation surfaces by proving Theorem~\ref{T3}. 

Since this theorem was established in Apisa~\cite{Apisa-mp1} for genus two translation surfaces in $\mathcal{H}(2)$ and $\mathcal{H}(1,1)$ whose $\GL_2(\R)$ orbit is dense, it suffices to consider genus two translation surfaces whose orbit closures are neither the entire stratum nor a locus of torus covers. These translation surfaces were classified by McMullen in~~\cite{Mc4} and~\cite{Mc5}; in particular, the orbit closures of these translation surfaces are exactly the collection of eigenforms for real multiplication by an order in a quadratic field. These loci are classified by the (non-square) discriminant $D$ of the order and have either one or two connected components depending on whether $D$ is not or is congruent to $1$ mod $8$ respectively.

We will now outline our approach to the finite blocking problem. By M\"oller~\cite[Theorem 2.6]{M2}, if $(X, \omega)$ is not a torus cover, then there is a unique  map $\piX:(X,\omega)\to \Xmin$ to a translation surface of minimal genus, and any map from $(X,\omega)$ to a translation surface is a factor of this map. In Apisa-Wright~\cite[Lemma 3.3]{Apisa-Wright}, it was shown that similarly, when $(X, \omega)$ is not a torus cover there is a quadratic differential $\Qmin$ with a degree 1 or 2 map $\Xmin\to \Qmin$ such that any map from $(X,\omega)$ to a quadratic differential is a factor of the composite map $\piQ:(X,\omega)\to \Qmin$. By Apisa-Wright~\cite[Theorems 3.5 and 3.15]{Apisa-Wright}, 

\begin{thm}[Apisa-Wright~\cite{Apisa-Wright}]
If $(X, \omega)$ is not a torus cover then 
\begin{enumerate}
\item If $p$ is a periodic point or zero, it can only be finitely blocked from other periodic points or zeros and a blocking set is contained in the collection of periodic points. 
\item If $p$ is not a periodic point, then it is only finitely blocked from other points in $B_1 := \piQ^{-1} \left(  \piQ(p) \right)$ and a blocking set is contained in the collection of periodic points along with the points in $B_1$ 
\end{enumerate}
\end{thm}

\noindent Since we are concerned with primitive translation surfaces $(X, \omega)$, which is equivalent to not being a torus cover in genus two, we have that $\Xmin = (X, \omega)$. Therefore, $\Qmin$ is the quotient of $(X, \omega)$ be the hyperelliptic involution. This shows the following:

\begin{prop}
If $(X, \omega)$ is a genus two translation surface that belongs to an eigenform locus then two points are finitely blocked only if they are exchanged by the hyperelliptic involution or they are both periodic points or zeros (in which case the blocking set is the collection of periodic points). 
\end{prop}

\begin{rem}\label{R:two-piers}
It is well-known that two points exchanged by the hyperelliptic involution are finitely blocked with a blocking set contained in the set of Weierstrass points (see for instance~\cite[Lemma 3.1]{Apisa-Wright}). Similarly, any Weierstrass point is blocked from itself by the other Weierstrass points. This reduces the proof of Theorem~\ref{T3} to showing that no other pairs of periodic points are finitely blocked by the set of periodic points. 
\end{rem}

\begin{lem}\label{L:convex}
Suppose that $\M$ is an orbit closure that is neither a locus of torus covers nor the golden eigenform locus and that contains a translation surface $(X, \omega)$ that may be represented as a convex octagon or decagon with opposite sides identified. Then Theorem~\ref{T3} holds for $\M$.
\end{lem}
\begin{proof}
Since Theorem~\ref{T3} holds for $\M$ a genus two stratum, suppose that $\M$ is a nonarithmetic eigenform locus different from the golden eigenform locus. On $(X, \omega)$, the Weierstrass points that are not singularities occur at midpoints of the edges of the $n$-gon and at the midpoint of the face of the $n$-gon. By Theorem~\ref{T2} the only periodic points are Weierstrass points and hence by Remark~\ref{R:two-piers} it suffices to show that the only pairs of periodic points and zeros that are finitely blocked by the collection of Weierstrass points are nonsingular Weierstrass points from themselves. The description of $(X, \omega)$ as a convex polygon with opposite sides identified and the explicit description of the nonsingular Weierstrass points makes this immediate. 
\end{proof}

\begin{cor}
Theorem~\ref{T3} holds in $\mathcal{H}(1,1)$. 
\end{cor}
\begin{proof}
By Leli\`evre and Weiss~\cite[Theorem 2]{LW}, every nonarithmetic eigenform locus (and trivially the decagon locus) contains a translation surface that may be represented as a convex decagon with opposite sides identified. By Lemma~\ref{L:convex}, Theorem~\ref{T3} holds for all surfaces in $\mathcal{H}(1,1)$ except possibly those in the golden eigenform locus. 

The golden point on the golden eigenform locus is not finitely blocked from itself because if it were then the blocking set could be taken to be the collection of periodic points in the golden eigenform locus, but we see in the left subfigure of Figure~\ref{F:golden-W} that these points do not block the horizontal line from the golden point to itself.  

To form the right subfigure, begin with the golden tetromino (the lefthand figure with $t = 0$) and horizontally shear so that the vertical cylinders become the dotted diagonal cylinders. Applying $\rho$ - the rel flow from the previous section - to the regular decagon we arrive at the righthand subfigure, where the golden points are on the vertical line passing through the midpoint of the decagon. 

\begin{figure}[h]
\centering
	\begin{tikzpicture}
		\draw (0,1) -- (1,1) -- (1,0) -- (2,0) -- (2,-2) -- (1,-2) -- (1,-1) -- (0,-1) -- (0,1);
		\draw[dotted] (0,0) -- (1,0) -- (1,-1) -- (2,-1);
		\node at (-1.75, .5) {$1-t$};
		\node at (-1.75, -.5) {$\phi^2+t$};
		\node at (-1, -1.5) {$2\phi-t$};
		\node at (.5, 1.5) {$1$};
		\node at (1.5, 1.5) {$\phi$};
		\draw[black, fill] (0,.5) circle[radius=1.5pt]; \node at (-.25,.5) {$1$};
		\draw[black, fill] (.5,.5) circle[radius=1.5pt]; \node at (.75,.5) {$2$};
		\draw[black, fill] (.5,-.5) circle[radius=1.5pt]; \node at (.25,-.5) {$3$};
		\draw[black, fill] (1.5,-.5) circle[radius=1.5pt]; \node at (1.25,-.5) {$4$};
		\draw[black, fill] (1.5,-1.5) circle[radius=1.5pt]; \node at (1.75,-1.5) {$5$};
		\draw[black, fill] (1,-1.5) circle[radius=1.5pt]; \node at (.75,-1.5) {$6$};
		\draw[black] (1, -1.25) circle[radius=3pt]; \draw[black] (1, -1.75) circle[radius=3pt]; 
	\end{tikzpicture}
	\qquad
	\begin{tikzpicture}
		\draw (0,0) -- (1.16, 0.38) -- (1.9, 1.38) -- (1.9, 2.62) -- (1.16, 3.62) -- (0, 4) -- (-1.16, 3.62) -- (-1.9, 2.62) -- (-1.9, 1.38) -- (-1.16, .38) -- (0,0);
		\draw[dashed] (1.16, 0.38) --  (-1.16, .38);
		\draw[dashed] (1.9, 1.38) -- (-1.9, 1.38);
		\draw[dashed] (1.16, 3.62) -- (-1.16, 3.62);
		\draw[dashed] (1.9, 2.62) -- (-1.9, 2.62);
		\draw[dotted] (-1.9, 1.38) -- (1.16, .38);
		\draw[dotted] (-1.9, 2.62) -- (1.9, 1.38);
		\draw[dotted] (-1.16, 3.62) -- (1.9, 2.62);
		\draw[black, fill] (-.58,3.81) circle[radius=1.5pt]; \node at (-.9, 4) {$2$};
		\draw[black, fill] (.58,3.81)  circle[radius=1.5pt]; \node at (.9, 4) {$1$};
		\draw[black, fill] (-1.53, 3.12) circle[radius=1.5pt]; \node at (-1.9, 3.12) {$4$};
		\draw[black, fill] (1.53, 3.12)  circle[radius=1.5pt]; \node at (1.9, 3.12) {$3$};
		\draw[black, fill] (-1.9, 2) circle[radius=1.5pt]; \node at (-2.2, 2) {$5$};
		\draw[black, fill] (0, 2) circle[radius=1.5pt]; \node at (.25, 2) {$6$};
		\draw[black] (0, 2.25) circle[radius=3pt]; 
		\draw[black] (0, 1.75) circle[radius=3pt];

	\end{tikzpicture}
\caption{Figure~\ref{fig:deformed-cyl} with all periodic points marked. Weierstrass points are marked as solid dots and the golden points as circles.}
\label{F:golden-W}
\end{figure}
Since any blocking set is contained in the set of periodic points we see that point $5$ is not blocked from any periodic point (consider the left half of the decagon) and similarly point $6$ is not blocked from any periodic point. Finally, consider the part of the decagon that lies on or above the horizontal line through the upper golden point. This subsurface is convex and all periodic points excluding points $5$, $6$, and the lower golden point lie on the boundary. Therefore, none of these point are finitely blocked from any other.
\end{proof}

\begin{proof}[Proof of Theorem~\ref{T3}:]
It remains to establish the claim in nonarithmetic eigenform loci in $\mathcal{H}(2)$. By M\"oller~\cite{M2}, the only periodic points in these loci are Weierstrass points. By Leli\`evre and Weiss~\cite[Theorem 1]{LW}, each nonarithmetic eigenform locus contains a generic translation surface that has a representation as a (not necessarily strictly) convex polygon with opposite sides identified except when $D = 5, 12, 17, 21$, $32, 41_0, 45, 77$ where the subscript denotes the spin component (to be explained shortly). By Lemma~\ref{L:convex}, this establishes the theorem outside of these eight cases. 

By McMullen~\cite[Theorem 3.3]{McM:spin} each of these eigenform loci contains the surface in Figure~\ref{fig:McMullen Prototype} where $b$, $c$, $e$ are any integers (with $b$ and $c$ strictly positive and $\mathrm{gcd}(b,c,e) = 1$) such that the following hold:
\begin{enumerate}
\item $\ds{ \lambda = \frac{e + \sqrt{D}}{2} }$  
\item $D = e^2 + 4bc$ 
\item $c + e < b$
\item If $D  = E f^2$, then the spin invariant is $\ds{ \frac{e-f}{2} + (c+1)b \text{ mod }  2 }$.
\end{enumerate}

\begin{figure}[h]
\centering
	\begin{tikzpicture}
		\draw (0,2) -- (2,2) -- (2,1) -- (1,1) -- (1,0) -- (0,0) -- (0,2);
		\draw[dashed] (0,1) -- (1,1) -- (1,2);
		\node at (-.5, .5) {$\lambda$};
		\node at (-.5, 1.5) {$c$};
		\node at (.5, 2.5) {$\lambda$};
		\node at (1.5, 2.5) {$b - \lambda$};
		\draw[black, fill] (.5,1.5) circle[radius=1.5pt]; \node at (.75,1.5) {$1$};
		\draw[black, fill] (1.5,1.5) circle[radius=1.5pt]; \node at (1.75,1.5) {$2$};
		\draw[black, fill] (1.5,1) circle[radius=1.5pt]; \node at (1.75,.75) {$3$};
		\draw[black, fill] (.5,.5) circle[radius=1.5pt]; \node at (.75,.75) {$4$};
		\draw[black, fill] (1,.5) circle[radius=1.5pt]; \node at (1.25,.25) {$5$};
	\end{tikzpicture}
\caption{McMullen's prototype with nonsingular Weierstrass points labelled. Opposite sides are identified. All angles are multiples of $\frac{\pi}{2}$. Edge labels are lengths.}
\label{fig:McMullen Prototype}
\end{figure}
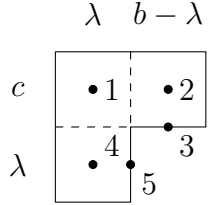

The moduli of the two horizontal cylinders are $1$ and $\frac{c}{b}$. The horizontal direction corresponds to a parabolic element of the Veech group. The parabolic element is just a product of Dehn twists about the core curves of the horizontal cylinders. If $\frac{c}{b} = \frac{c'}{b'}$ where $c'$ and $b'$ are coprime, then the Dehn twist is a complete $c'$-twist about the top cylinder and a complete $b'$ twist about the lower cylinder. The element of the Veech group is $\begin{pmatrix} 1 & b' \\ 0 & 1 \end{pmatrix}$. Letting this Dehn multi twist act on the Weierstrass points, (and hence thinking about it as an element of a $\mathrm{Sym}(5)$) we see that the action is $(12)^{c'}(45)^{b'}$.

Similarly, in the vertical direction the modulus of the leftmost cylinder is $\frac{\lambda}{\lambda + c}$ and the modulus of the rightmost cylinder is $\frac{b-\lambda}{c}$. Denote these two cylinders by $V_\ell$ and $V_r$ where the subscript denotes right and left. 
\[ \frac{\mathrm{Mod}(V_r)}{\mathrm{Mod}(V_\ell)} = \frac{(b-\lambda)(\lambda + c)}{c \lambda} \]
Recall that $\lambda^2  = e\lambda + bc$, so
\[ \frac{\mathrm{Mod}(V_r)}{\mathrm{Mod}(V_\ell)} = \frac{b-c-e}{c} \]
Express this ratio in lowest terms as $\frac{b''}{c''}$. The corresponding Dehn multi twist is a complete $b''$ twist in the rightmost cylinder and a complete $c''$ twist in the leftmost. Letting this Dehn multi twist act on the Weierstrass points, (and hence thinking about it as an element of a $\mathrm{Sym}(5)$) we see that the action is $(23)^{c''}(14)^{b''}$. 

\begin{lemma}
If $c'$ and $b'$ or $c''$ and $b''$ are both odd, then no two distinct Weierstrass points are blocked from each other on $(X, \omega)$.
\end{lemma}
\begin{proof}
First, we see that the Weierstrass point marked $1$ is not finitely blocked from any of the other Weierstrass points. If $c', c'', b',$ and $b''$ are all odd, then the Dehn twists in the Veech group that we identified act on the Weierstrass point by the group generated by $(23)(14)$ and $(12)(45)$. Since this group is transitive, any Weierstrass point can be moved to the Weierstrass point marked $1$ and hence no Weierstrass point is blocked from any other distinct Weierstrass point.

Now suppose that $c'$ and $b'$ are both odd, but that either $c''$ or $b''$ is even. In this case, the Dehn twists correspond to $(12)(45)$ and $(23)$ or $(14)$. In the first case, there are two orbits $\{1,2,3\}$ and $\{4,5\}$. In the second, there are also two orbits $\{1,2,4,5\}$ and $\{3\}$. In the first case, $\{1,2,3\}$ all illuminate all other distinct Weierstrass points. Since $\{4,5\}$ illuminate each other we are done. In the second case, $\{1,2,4,5\}$ illuminate all other distinct Weierstrass points so we are done. 

The case where $c''$ and $b''$ are both odd and either $c'$ or $b'$ is even is completely analogous. 
\end{proof}

Therefore, to complete the proof we must find for the $D$ listed above three integers $(b,c,e)$ with $b$ and $c$ strictly positive and such that:
\begin{enumerate}
\item $\mathrm{gcd}(b,c,e) = 1$
\item $c + e < b$
\item $D = e^2 + 4bc = E f^2$ for $E$ and $f$ integers and $f >0$.
\item Either $\frac{c}{b}$ or $\frac{b-c-e}{c}$ when put in lowest terms is a ratio of two odd integers.
\item If the spin parity is specified, then $\frac{e - f}{2} + (c+1)b \text{ mod } 2$ is the desired parity.
\end{enumerate}

\begin{figure}[h!]
\begin{tabular}{m{1.5cm}m{3cm}m{3cm}m{3cm}}
\toprule
$D=$ & $(b,c,e)=$ & $b-c-e =$ & Ratio of moduli \\
\midrule
5 & (1,1,-1) & 1 & $\frac{c}{b} = \frac{1}{1}$  \\
12 & (3,1,0) & 2 & $\frac{c}{b} = \frac{1}{3}$ \\
$17_1$ & (1,2,-3) & 2 & $\frac{c}{b-c-e} = \frac{1}{1}$ \\
21 & (1,3,-3) & 1 & $\frac{c}{b-c-e} = \frac{3}{1}$ \\
32 & (4,1,-4) & 7 & $\frac{c}{b-c-e} = \frac{1}{7}$ \\
45 & (1,5,-5) & 1 & $\frac{c}{b-c-e} = \frac{5}{1}$ \\
77 & (1,7,-7) & 1 & $\frac{c}{b-c-e} = \frac{7}{1}$  \\
\bottomrule
\end{tabular}
\caption{Triples for $D \ne 41_0$}
\label{F:table-triples}
\end{figure}

\noindent In the case of $17_1$ we must also check the spin parity, i.e.
\[ \frac{e-f}{2} + (c+1)b = -2 + (2+1) \cdot 1 \equiv 1 \text{ mod } 2 \]

We must analyze $D = 41_0$ separately, see Figure~\ref{fig:Prototype41-0}. The prototype in Figure~\ref{fig:Prototype41-0} and the specified lines show that Weierstrass points $1$, $2$, and $5$ are not finitely blocked from any other distinct Weierstrass points (we omit this computation). 
\begin{figure}[h]
\centering
	\begin{tikzpicture}
		\draw (0,2) -- (5,2) -- (5,0) -- (2.7015, 0) -- (2.7015, -2.7015) -- (0,-2.7015) -- (0,2);
		\draw[dotted] (0,0) -- (2.7015,0) -- (2.7015 , 2);
		\draw[black, fill] (0,-1.3507) circle[radius=1.5pt]; 
		\draw[black, fill] (1.3507,-1.3507) circle[radius=1.5pt]; 
		\draw[black, fill] (3.85,1) circle[radius=1.5pt]; 
		\draw[black, fill] (3.85,2) circle[radius=1.5pt]; 
		
		\draw[dashed] (0,-1.3507) -- (3.85,1);
		\draw[dashed] (0,-1.3507) -- (3.85,2);
		\draw[dashed] (1.3507,-1.3507) -- (3.85,2);
	\end{tikzpicture}
\caption{Prototype for $D = 41_0$ with $(b,c,e) = (5,2,-1)$}
\label{fig:Prototype41-0}
\end{figure}
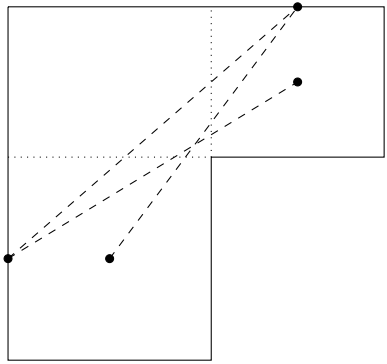
The horizontal shear yields the permutation $(45)$ and the vertical shear yields the permutation $(23)$. This shows that the Weierstrass points $3$ and $4$ can be moved by an element of $\GL(2, \R)$ to Weierstrass points that are not finitely blocked from other distinct Weierstrass points. 


\end{proof}

%
%

\section{Proof of Theorem~\ref{T:criterion} - Hyperelliptic Curves and Holomorphic $d$-differentials}~\label{S:HyperBilliards}

In this section we will determine exactly which holomorphic $d$-differentials unfold to hyperelliptic Riemann surfaces. We want our discussion to apply to billiards and so we use the approach of Mirzakhani-Wright~\cite[Section 6]{MirWriRank}. Given an $n$-gon with connected boundary and with angles, specified in clockwise order, $\theta = \left( \frac{a_1 }{b_1}, \hdots, \frac{a_n }{b_n} \right) \cdot \pi$ where $a_i$ and $b_i$ are coprime pairs of integers for all $i$,  let $\M(\theta)$ be the smallest affine invariant submanifold containing all unfoldings of rational polygons with angles (ordered clockwise) $\theta$. Let $d$ to be the least common multiple of $\{ b_i \}_{i=1}^n$.


\begin{lemma}\label{L:Q-reduction}
If $\theta$ and $\theta'$ are permutations of each other, then $\M(\theta)$ and $\M(\theta')$ coincide.
\end{lemma}
\begin{proof}
An unfolding is completely specified by its pillowcase double, which is a $d$-differential on the Riemann sphere with all zeros on the real line. Suppose that $P$ belongs to a stratum $\mathcal{Q}$ of $d$-differentials and let $\mathcal{U}(\theta)$ be the collection of all such pillowcase doubles. Given the pillowcase double $P$, all zeros may be moved freely on the real line by altering the lengths of the sides of $B$. This shows that $\mathcal{U}(\theta)$ is Zariski dense in $\mathcal{Q}$. In particular, the Zariski closure of $\mathcal{U}(\theta)$ contains $\mathcal{U}(\theta')$ since all strata of $d$-differentials on the sphere are connected.

Let $\pi: \mathcal{Q} \ra \mathcal{H}$ be the map that associates to a $d$-differential its canonical unfolding as an abelian differential. This map is holomorphic. By Filip~\cite{Fi1}, $\M(\theta)$, which is the orbit closure of $\pi \left( \mathcal{U}(\theta) \right)$, is a variety. In particular, $\M(\theta)$ contains the image of the Zariski closure of $\mathcal{U}(\theta)$ under $\pi$. Therefore, $\M(\theta)$ contains $\M(\theta')$. By symmetry these two orbit closures coincide. 
\end{proof}

In light of the lemma, we will relax our assumption that $\theta$ be specified up to cyclic order and instead insist that it is specified as a set. By Lemma~\ref{L:Q-reduction}, the loci $\M(\theta)$ are precisely the smallest affine invariant submanifolds containing the unfoldings of all $d$-differentials with the same singularity type as the pillowcase double of a polygon with angles $\theta$. 

Given a $d$-differential on the sphere, it canonically unfolds to a flat surface that is invariant under a $\Z/d\Z$ action, see~\cite[Section 2]{Backgammon-k}. The quotient of the flat surface by the cyclic group is exactly the original $d$-differential. If the $d$-differential has singularities $\left( \frac{a_1 \pi}{b_1}, \hdots, \frac{a_n \pi}{b_n} \right) 2\pi$ where $d = \mathrm{lcm}\left(b_1, \hdots, b_n \right)$ then a singularity of cone angle $\frac{a_i}{b_i} \cdot 2\pi$ has $\frac{d}{b_i}$ preimages each $b_i$-ramified.

\begin{lemma}\label{L:d-hyp}
Suppose that $P$ is a $d$-differential on the sphere with cone-angles $\left( \frac{a_1}{b_1}, \hdots, \frac{a_n}{b_n} \right) \cdot 2\pi$. The deck group of the canonical covering contains a hyperelliptic involution if and only if $d := \mathrm{lcm}\left( b_1, \hdots, b_n \right) = 2k$ for some integer $k$ and (up to permutation) the cone angles are:
\[ \left( \frac{a_1}{d}  , \frac{ a_2}{d} , \frac{ a_3}{2} , \hdots, \frac{ a_n}{2}  \right) 2\pi \]
where $a_3, \hdots, a_n$ are odd, $a_1$ and $a_2$ are positive integers coprime to $k$, and $k(n-2) + m = 2g+2$ where $m$ is the number of odd integers in $\{a_1, a_2 \}$ and $g$ the genus of the canonical cover.
\end{lemma}
\begin{proof}
Suppose first that the deck group contains the hyperelliptic involution. The deck group is then a cyclic group of order $d = 2k$, since it contains an involution. Quotienting by the hyperelliptic involution gives a quadratic differential $(Q, q)$ on the sphere with $2g+2$ singularities of cone angle an odd multiple of $\pi$. The deck group descends to a cyclic group of order $k$ of automorphisms that preserve $(Q, q)$. Up to M\"obius transformations, the map from $(Q, q)$ to $P$ must be given by $f(z) = z^k$. Since $P$ pulls back to a quadratic differential under $f(z) = z^k$ it follows that the singularities at $0$ and $\infty$ have the form $\frac{a_i}{2k} \cdot 2\pi$ for some integer $a_i$ ($i = 1, 2$) and that all other singularities have the form $\frac{a_i}{2} \cdot 2 \pi$ for some integer $a_i$ ($i>3$). Since the unfolding map is canonical, $a_1$ and $a_2$ are coprime to $k$. Pulling back the $d$-differential to a quadratic differential on the sphere gives $k(n-2) + m$ points with odd cone angle, which must be Weierstrass points under the double cover. This implies that $k(n-2) + m = 2g+2$. 


For the reverse direction, the canonical cover by a quadratic differential must be the cyclic cover which is $k$ ramified over $\frac{a_1}{d}$ and $\frac{ a_2}{d}$. The canonical cover of this quadratic differential is a hyperelliptic abelian differential $(X, \omega)$ and the map from $(X, \omega)$ to $P$ has the same branching as the canonical cover of $P$; therefore this cover must be the canonical one.
\end{proof}

\begin{cor}\label{C:d-hyp}
The billiard tables whose pillowcase doubles are as in Lemma~\ref{L:d-hyp} and that have a genus two unfolding are the ones listed below.
\end{cor}
\begin{proof}
When $k$ is even, the condition that $a_1$ and $a_2$ are coprime to $k$ implies that $m = 2$ and so we have that $k(n-2) = 4$. The possible combinations are $(k,n) = (2, 4), (4, 3)$ and the possible singularities of $P$ are 
\[ \left( \frac14, \frac34, \frac12, \frac12 \right)\pi \qquad \left( \frac18, \frac38, \frac12 \right)\pi \]
When $k=1$, the only possibility is
\[ \left( \left( \frac12 \right)^5, \frac32 \right)\pi \]
When $k > 1$ is odd, we have that $k(n-2) + m = 6$ and so $(k,m,n) = (5, 1, 3); (3, 0, 4)$ and the tables are
\[  \left( \frac{1}{10}, \frac{4}{10}, \frac{1}{2} \right)\pi \qquad \left( \frac{2}{10}, \frac{3}{10}, \frac{1}{2} \right)\pi \qquad \left( \frac{2}{6}, \frac{4}{6}, \frac{1}{2}, \frac{1}{2} \right)\pi \]

\end{proof}

\begin{lemma}\label{L:d-hyp2}
Let $P$ be a $d$-differential on the sphere. The canonical cover contains a hyperelliptic involution that is not part of the deck group if and only if $B$ is not the double of a rectangle or the $45-45-90$ triangle and the singularities are either $\left( \frac{a_1}{d}  , \frac{ a_1}{d},  \frac{a_2}{d}  , \frac{ a_2}{d} \right) 2\pi$ or $ \left( \frac{a_1}{d}  ,  \frac{a_1}{d}, \frac{a_2}{d} \right) 2\pi$ where $a_i$ are positive integers so that $a_1$ is coprime to $d$. 
\end{lemma}
\begin{proof}
Suppose first that the canonical cover contains a hyperelliptic involution $\iota$ that is not part of the deck group. Let $J: (X, \omega) \ra \left( \mathbb{C}\mathrm{P}^1, q \right)$ be the quotient of $(X, \omega)$ by $\iota$; in other words, $\left( \mathbb{C}\mathrm{P}^1, q\right)$ is a half-translation structure on the sphere whose holonomy double cover is $(X, \omega)$. The hyperelliptic involution then descends to an involution that acts by $-1$ on the $d$-differential $P$. The quotient of $P$ by the involution gives a double branched cover $\pi: P \ra P'$ to a $2d$-differential on the sphere $P'$. Since $\iota$ commutes with the action of the deck group of the canonical cover $f: (X, \omega) \ra P$ it follows that we have the following commutative diagram, shown in Figure~\ref{F:cd1}.

\begin{figure}[h]
	\begin{tikzcd}
		P \arrow[d, "\pi"] & (X, \omega) \arrow[l, "f"] \arrow[d, "J"] \\
		P' & \left( \C\mathrm{P}^1, q \right) \arrow[l, "z^d"]
	\end{tikzcd}
\caption{The maps between the differentials}
\label{F:cd1}	
\end{figure}

\noindent The involution on $P$ is a holomorphic involution on the sphere and hence fixes exactly two points. From the commutative diagram we have that the branch locus of $f$ contains only the preimages of $0$ and $\infty$ under $\pi$. This implies that the branch locus of $f$ contains at most four points. Since the branch locus of $f$ is the collection of singularities of $P$, it follows that $P$ and $P'$ have at most four singularities. The singularities of $P'$ are therefore
\[ \left( \frac{a_1}{2d}, \frac{a_2}{2d}, \frac{a_3}{2}, \frac{a_4}{2} \right) 2\pi \qquad \text{or} \qquad \left( \frac{b_1}{2d}, \frac{b_2}{2d}, \frac{b_3}{2}\right) 2\pi \]
where the first two points lie at $\{0, \infty\}$. Notice that the branch locus in $P'$ must exclusively contain points in $\{0, \infty\}$ or singularities because $J$ is only branched over singularities (and hence the statement about $\pi$ follows from the commutative diagram). Since the sum of the cone angles must be $(n-2)\pi$ where $n$ is the number of singularities we see that $a_3, a_4,$ and $b_3$ are odd integers (we just argued that these numbers are not $2$ and the next observation shows that they all must be less than $4$). Therefore, $J$ is branched over the preimages of these points and hence $\pi$ is branched over them as well. The singularities of $P$ are therefore, respectively,
\[ \left( \frac{a_1}{2d}, \frac{a_1}{2d}, \frac{a_2}{2d}, \frac{a_2}{2d} \right) \qquad \text{or} \qquad \left( \frac{b_1}{2d}, \frac{b_1}{2d}, \frac{b_2}{d} \right) \]
Since each vertex has at most $d$ preimages under $f$ we see that, perhaps after choosing new integers $a_i$ and $b_i$, the singularities of $P$ are 
\[ \left( \frac{a_1}{d}, \frac{a_1}{d}, \frac{a_2}{d}, \frac{a_2}{d} \right) \qquad \text{or} \qquad \left( \frac{b_1}{d}, \frac{b_1}{d}, \frac{b_2}{d} \right) \]
and the singularities of $P'$ are 
\[ \left( \frac{a_1}{d}, \frac{a_2}{d}, \frac{1}{2}, \frac{1}{2} \right) 2\pi \qquad \text{or} \qquad \left( \frac{b_1}{d}, \frac{b_2}{2d}, \frac{1}{2}\right) 2\pi \]
where the first two singularities lie at $\{0, \infty\}$ and $\pi$ is branched over the last two. The first two singularities in $P'$ have at most two preimages on $(X, \omega)$ and hence $a_1$ and $b_1$ are coprime to $d$.


For the reverse direction, both of the $d$-differentials described have an involution so that the quotient map $\pi: P \ra P'$ sends $P$ to a $2d$-differential $P'$ with singularities $\theta' = \left( \frac{a_1}{d}  , \frac{ a_2}{d},  \frac{1}{2}, \frac{1}{2} \right) 2\pi$ or $\theta' =  \left( \frac{a_1}{d}  , \frac{ a_2}{2d},  \frac{1}{2} \right)$ where $a_1$ is a positive integer coprime to $d$. 

Suppose first that $d$ is odd. Then the abelian differential that canonically covers $P'$ is a degree $2d$ cover of $P'$ since this is the least common multiple of the denominators. Since $(X, \omega)$ is a degree $2d$ cover of $P'$, it must be the canonical cover. By Lemma~\ref{L:d-hyp}, $\M(\theta')$ is contained in a hyperelliptic locus. Since $\M(\theta) = \M(\theta')$, $\M(\theta)$ is contained in a hyperelliptic locus and the hyperelliptic involution is not contained in the deck group since the deck group has odd order.

Suppose now that $d=2k$ for some integer $k$. The canonical covering $g: (Y, \eta) \ra P'$ is also degree $d$. Since $\pi \circ f: (X, \omega) \ra P'$ is a holomorphic map that pulls back $P'$ to $(X, \omega)$ it follows that there is a map $p: (X, \omega) \ra (Y, \eta)$ so that $\pi \circ f = p \circ g$. By Lemma~\ref{L:d-hyp} it follows that $g$ factors through the hyperelliptic involution $J'$ and the map $f(z) = z^k$ on the sphere. The situation is summarized in the commutative diagram in Figure~\ref{F:cd2}

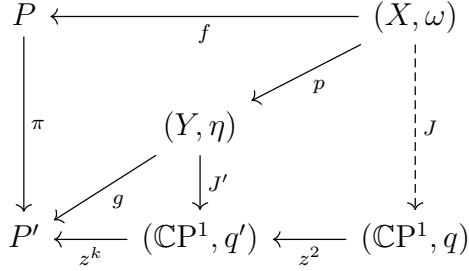
\begin{figure}[h]
	\begin{tikzcd}
		P \arrow[dd, "\pi"] & & (X, \omega) \arrow[ll, "f"] \arrow[dashed, dd, "J"] \arrow[dl, "p"] \\
		  & (Y, \eta) \arrow[dl, "g"] \arrow[d, "J'"] &  \\
		P' & \left( \C\mathrm{P}^1, q' \right) \arrow[l, "z^k"] & \left( \C\mathrm{P}^1, q \right) \arrow[l, "z^2"]
	\end{tikzcd}
\caption{The maps between the differentials when $d$ is even}
\label{F:cd2}	
\end{figure}

\noindent We wish to produce a map $J$ that makes the diagram commute. Since we have that $p$ is a translation covering we see that $p$ is unbranched when $n = 4$ and when $n = 3$ and $k$ is even. When $n = 3$ and $k$ is odd, $p$ is branched over the two preimages of the vertex of angle $\frac{a_1}{d}$ on $P'$. Covering space theory implies that $J'$ lifts to the desired map $J: (X, \omega) \ra \left( \C \mathrm{P}^1, q \right)$. Therefore, $(X, \omega)$ is hyperelliptic.

Suppose now that $f$ factors though $J$. If it did then, after a M\"obius transformation, the map $h(z) = z^k$ would satisfy $f = h \circ J$ and the pullback of $P$ under $h$ would be $(\mathbb{C}\mathrm{P}^1, q)$, the quotient of $(X, \omega)$ by $J$. In particular, this would mean that two of the angles on $P$ were $\frac{\pi}{2}$. This means that either $P$ is a rectangle or the triangle $\left( \frac{1}{2}, \frac{1}{4}, \frac{1}{4} \right)$. 
\end{proof}

\begin{rem}
If the singularities of $P$ are $\left( \frac{a_1}{d}, \frac{a_1}{d}, \frac{a_2}{d}, \frac{a_2}{d} \right) 2\pi$ then $P$ unfolds to a genus $d-1$ surface in $\mathcal{H}(a_1-1, a_1 - 1, a_2 - 1, a_2 - 1)$. If the singularities of $P$ are $\left( \frac{a_1}{d}, \frac{a_1}{d}, \frac{a_2}{d} \right)2\pi$ then when $d$ is even the unfolding is a genus $\frac{d-2}{2}$ surface in $\mathcal{H}\left(a_1 - 1, a_1 - 1, \frac{a_2}{2} - 1, \frac{a_2}{2} - 1 \right)$ and when $d$ is odd the unfolding is a genus  $\frac{d-1}{2}$ surface in $\mathcal{H}(a_1-1, a_1-1, a_2-1)$
\end{rem}

From the remark we have the following immediate corollary:

\begin{cor}\label{C:d-hyp2}
The surfaces in Lemma~\ref{L:d-hyp2} have a genus two unfolding only when the singularities of $P$ are 
\[ \left( \frac13, \frac13, \frac23, \frac23 \right)2\pi \qquad \left( \frac15, \frac15, \frac35 \right)2\pi \qquad \left( \frac25, \frac25, \frac15 \right)2\pi \qquad \left( \frac16, \frac16, \frac46 \right)2\pi \]
\end{cor}

Notice that Theorem~\ref{T:criterion} is immediate from Lemmas~\ref{L:d-hyp} and~\ref{L:d-hyp2}. 

%
%

\section{Proof of Theorem~\ref{T:final} - Billiard Tables Unfolding to Genus Two Surfaces}~\label{S:Genus2} 

We will now solve the finite blocking problem for rational polygons that unfold to genus two translation surfaces. 

\begin{defn}
A rational polygon that unfolds to a hyperelliptic Riemann surface will be called ``special" if the group of deck transformations does not include the hyperelliptic involution
\end{defn}

The list of special rational polygons is given in Lemma~\ref{L:d-hyp2}. 

\begin{prop}\label{P:FBP2}
Suppose that $B$ is a polygon with connected boundary with rational angles $\theta = \left( \frac{a_1}{b_1}, \hdots, \frac{a_n}{b_n} \right) \cdot \pi$ that unfolds to a hyperelliptic Riemann surface. Suppose that any point $p$ on the unfolding of $B$ is finitely blocked from no points if it is a zero and only finitely blocked from its image under the hyperelliptic involution otherwise. Define $d:= \mathrm{lcm}(b_1, \hdots, b_n)$. The following is a complete list of the pairs of finitely blocked points on $B$:
\begin{enumerate}
\item If $B$ is not special any vertex of angle $\frac{\pi}{d}$ is finitely blocked from itself.
\item If $B$ is special any two distinct vertices of angle $\frac{\pi}{d}$ are finitely blocked from each other.
\item If $B$ is an isosceles triangle and the non-repeated angle has the form $\frac{\pi}{d}$, then this vertex is finitely blocked from itself.
\end{enumerate}
\end{prop}
\begin{proof}
Suppose first that $p$ and $q$ are two distinct points on $B$ that are finitely blocked from each other. Let $\wt{p}$ be a preimage of $p$ on the unfolding of $B$. Since $\wt{p}$ is finitely blocked from at most one other point it follows that there is exactly one preimage of $q$ on the unfolding of $B$. By symmetry of hypotheses, there must also be exactly one preimage of $p$ as well. A point on $B$ has exactly one preimage on the unfolding if and only if it corresponds to a vertex of angle $\frac{\pi}{d}$. There are two vertices of cone angle $\frac{\pi}{d}$ only if $d  = 2$, $d = 4$, or $B$  is special. If $B$ is not special then the preimage of $p$ is a Weierstrass point and hence only finitely blocked from itself, a contradiction since it is also blocked from the preimage of $q$. Therefore, $B$ is special and the two vertices of angle $\frac{\pi}{d}$ unfold to two points exchanged by the hyperelliptic involution, which are hence finitely blocked from each other. 

Suppose now that $p$ is a point on $B$ that is finitely blocked from itself. By hypothesis, $\wt{p}$ must contain one non-singular Weierstrass point and hence $p$ is a vertex of angle $\frac{\pi}{d}$. Such a point becomes a nonsingular Weierstrass point when $B$ is not special or if $B$ is an isosceles triangle and the vertex is the one corresponding to the non-repeated angle.
\end{proof}

\begin{proof}[Proof of Theorem~\ref{T:final}]
The list of rational polygons with connected boundary that unfold to a genus two translation surface in Figure~\ref{F:table} is complete by Corollaries~\ref{C:d-hyp} and~\ref{C:d-hyp2}.

We begin with the claim about when an unfolding is a torus cover. By Apisa-Wright~\cite[Theorem 1.1]{Apisa-Wright} (see also Proposition 3.10 in Apisa-Wright~\cite{Apisa-Wright}), if a rational polygon unfolds to a torus cover it in fact unfolds to a square tiled surface. This implies that the $\Q$-linear span of the relative periods is a two-dimensional $\Q$-vector space. Therefore, the ratio of lengths $\frac{x_1}{x_2}$ and $\frac{y_1}{y_2}$ must be rational. For the converse, we observe that when these ratios are rational we may scale the figure so that the relative periods of the unfolding lie in either the Gaussian or Eisenstein integers, which implies that the unfolding is a square-tiled surface.

The claim about finite blocking for polygons that unfold to genus two translation surfaces follows from Proposition~\ref{P:FBP2} and Theorem~\ref{T3} (which certifies that the hypotheses in the statement of the proposition are satisfied). 

The computation of $\M(\theta)$ in Figure~\ref{F:table} follows in all cases from either the definition of the given locus or from the ``trivial rank bound" in Mirzakhani-Wright~\cite[Lemma 7.1]{MirWriRank}, which states that if there is some entry of $\theta$ that is not an integer multiple of $\frac{\pi}{2}$, then $\M(\theta)$ has rank at least $|\theta|-2$ where $|\theta|$ is the length of $\theta$. In genus two strata of translation surfaces the only affine invariant submanifold that has rank two is the entire stratum itself.
\end{proof}

%
%

\section{Similar Cylinders}~\label{S:SimilarCylinder}

It remains to show that there is at most one non-arithmetic rank two affine invariant submanifold in $\mathcal{H}(6)$. Before proving that statement, we first prove an independently interesting cylinder lemma. The idea for the lemma arose from conversations with Alex Wright. 

We begin by summarizing some facts about cylinder deformations. Fix an affine invariant submanifold $\M$. By Avila-Eskin-M\"oller~\cite{AEM}, the projection of the tangent space of $\M$ to absolute cohomology is complex symplectic. The rank of $\M$ is defined to be half the complex dimension of this projection. The affine invariant submanifold is said to have $d$ dimensions of rel, if $d = \dim_\C \M - 2r$ where $r$ is the rank of $\M$.

On a translation surface in $\M$, two cylinders are said to be equivalent if on all nearby translation surfaces in $\M$ their core curves are parallel. A maximal set of equivalent cylinders is called an equivalence class of cylinders. 

\begin{defn}
Fix an equivalence class of cylinders with core curves $\{\gamma_i\}_{i=1}^n$ - all of whose periods differ up to scaling by a positive real number - and heights $\{h_i\}_{i=1}^n$. The standard shear, which is defined up to scaling by a nonzero complex scalar, refers to the cohomology class $\sum_{i=1}^n h_i \gamma_i^*$ where $\gamma_i^*$ is the cohomology class corresponding to intersections with $\gamma_i$ for $i = 1, \hdots, n$.
\end{defn}

By Wright~\cite[Theorem 1.1]{Wcyl}, the standard shear always belongs to the tangent space of $\M$. By Wright~\cite[Theorem 1.10]{Wcyl}, if $\M$ is a rank $r$ affine invariant submanifold then there is always a translation surface $(X, \omega)$ in $\M$ that is horizontally periodic, with $r$ horizontal equivalence classes, and so that the standard shears of the horizontal equivalence classes project to a Lagrangian subspace in the projection of $T_{(X, \omega)} \M$ to absolute cohomology. These properties are actually implied by a stronger condition, i.e. that the ``twist space and cylinder preserving space of $(X, \omega)$ coincide", see Wright~\cite{Wcyl} for a definition.

\begin{defn}
Two cylinders $C_1$ and $C_2$ on a translation surface are similar if the two cylinders, with boundaries marked at cone points, are rescaled versions of each other, i.e. if there is some positive real number $\lambda$ so that $C_1 = \begin{pmatrix} \lambda & 0 \\ 0 & \lambda \end{pmatrix} C_2$. 
\end{defn}

The aim of this section is to show that if two simple cylinders form an equivalence class on a translation surface whose orbit closure has no rel, then the two cylinders are similar. 

\begin{ass}
Let $\M$ be an affine invariant submanifold with no rel and that has rank $r > 1$. 
\end{ass}

\begin{lemma}\label{L:EC}
Let $\cC$ be an $\M$-equivalence class of cylinders on a translation surface $(X, \omega) \in \M$. Let $U$ be a connected neighborhood of $(X, \omega)$ in $\M$ on which the cylinders in $\cC$ persist (i.e. their heights never reach zero). Then for each surface in $U$, $\cC$ remains an $\M$-equivalence class of cylinders.
\end{lemma}
\begin{proof}
Suppose to a contradiction that there is a point $(Y, \eta)$ in $U$, in which $\cC$ fails to be an equivalence class, i.e. there is a superset $\cC'$ of $\cC$ that is an equivalence class. By Mirzakhani-Wright~\cite[Corollary 1.5]{MirWri}, the ratio of moduli of two $\M$-equivalent cylinders is locally constant in $\M$. Let $\gamma$ be a path in $U$ from $(Y, \eta)$ to $(X, \omega)$. Since the ratio of moduli of cylinders in $\cC'$ is constant, if one cylinder fails to persist then all of them must. Some cylinders must fail to persist along this path since $\cC$ is an equivalence class at $(X, \omega)$. However, by assumption the cylinders in $\cC$ persist along the path and so we have a contradiction. 
\end{proof}

Suppose that $(X, \omega)$ is a translation surface in $\M$ that is horizontally periodic and that has cylinder preserving space and twist space coinciding. By Avila-Eskin-M\"oller~\cite{AEM}, the projection of the tangent space of $\M$ to absolute cohomology is complex symplectic. Since $\M$ has no rel, for every equivalence class $\cD$ of horizontal cylinders on $(X, \omega)$ there is a tangent vector $v_{\cD}$ which is nonzero on the core curve of every cylinder in $\cD$ and which vanishes on the core curve of every other horizontal cylinder in $(X, \omega)$. There is also the standard shear $\sigma_{\cD}$ which is supported on $\cD$. Notice that 
\[ B:=\{ \sigma_{\cD}, v_{\cD} : \text{ $\cD$ a horizontal equivalence class on $(X, \omega)$ } \} \]
is a basis of the tangent space of $\M$ at $(X, \omega)$. Let $V$ be an open connected neighborhood of $(X, \omega)$ on which all horizontal cylinders and all vectors in $B$ remain well-defined tangent directions.

\begin{lemma}\label{L:similar1} 
Suppose that $(X, \omega)$ is a horizontally periodic translation surface in $\M$ with twist space and cylinder preserving space coinciding. Suppose that $\cC$ is an $\M$-equivalence class of horizontal cylinders on $(X,\omega)$ that contains exactly two cylinders. Let $V$ be as above. If the cylinders in $\cC$ are simple for some $(Y, \eta) \in V$, then the cylinders in $\cC$ are similar on $(Y, \eta)$. 
\end{lemma}
\begin{proof}
Let $\cC = \{C_1, C_2\}$ and suppose to a contradiction that they are not similar on $(Y, \eta)$. Since they are not similar, on $(Y, \eta)$ use the standard shear to shear $\cC$ (which is still an equivalence class on $(Y, \eta)$ by Lemma~\ref{L:EC}) so that one cylinder, say $C_1$, has a vertical saddle connection $v$, but $C_2$ does not. Let $(Y, \eta)$ now denote this new surface. 

Apply the standard deformation to $\cC$ to send the period of $v$ to zero while only altering imaginary parts of periods. Let $(Z, \zeta)$ be the resulting boundary translation surface. Since only the height of one simple cylinder vanished, $(Z, \zeta)$ is connected. Let $\cN$ be the largest affine invariant submanifold in the boundary of $\cM$ containing $(Z, \zeta)$ (the boundary is defined in Mirzakhani-Wright~\cite{MirWri}). 

Let $\{ \cD_i \}_{i=1}^r$ be the horizontal $\M$-equivalence classes on $(X, \omega)$. Suppose too that $\cD_1 = \cC$. In the basis $\{ \sigma_{\cD_i}, v_{\cD_i} \}_{i=1}^r$ of the tangent space of $\M$ at $(Y, \eta)$ the only deformations that fail to fix $\cC$ are $\sigma_{\cC}$ and $v_{\cC}$. By Mirzakhani-Wright~\cite[Theorem 2.7]{MirWri}, the span of $\{ \sigma_{\cD_i}, v_{\cD_i} \}_{i=2}^r \cup \{ v_\cC \}$ injects into the tangent space of $\cN$ at $(Z, \zeta)$. Since all cylinders in $\bigcup_{i \ne 1} \cD_i$ persist (and are disjoint) on $(Z, \zeta)$, we observe that the image of $v_{\cD_i}$ on $(Z, \zeta)$ (for $i \ne 1$) causes cylinders in $\cD_i$ to have the period of their core curves change while fixing the periods of the core curves of all other cylinders in $\bigcup_{i \ne 1} \cD_i$. 

Recall that the rank of $\cN$ is the maximum dimension of the $\Q$ linear span of the periods of the core curves of a collection of disjoint cylinders on translation surfaces in $\cN$. Therefore, we see that the rank of $\cN$ is at least $r-1$. By Mirzakhani-Wright~\cite[Corollary 2.8]{MirWri} the rank of $\cN$ is strictly smaller than the rank of $\cM$ so the rank of $\cN$ is exactly $r - 1$. Therefore, $\{ \sigma_{\cD_i}, v_{\cD_i} \}_{i=2}^r$ forms a basis of the projection of $T_{(Z, \zeta)} \cN$ to absolute cohomology, taken as a complex vector space. 

There is a vector $v$ in the span of $\{ \sigma_{\cD_i}, v_{\cD_i} \}_{i=2}^r$ so that $v_{\cC} - v$ is rel. Let $\gamma$ be the image of the core curve of $C_2$ on $(Z, \zeta)$. Since $C_2$ did not contain a vertical saddle connection, $\gamma$ remains an absolute homology class on $(Z, \zeta)$. Moreover, $v$ preserves $\gamma$ while $v_{\cC}$ changes its period. Therefore, $v_{\cC} - v$ changes the period of an absolute homology class and hence cannot be rel, which is a contradiction. 
\end{proof}

\begin{lemma}[Similar Cylinder Lemma]\label{L:SCL}
Suppose that $\M$ is an affine invariant submanifold that has no rel. Let $(Y, \eta)$ be a translation surface with two simple cylinders $C_1$ and $C_2$ that form an equivalence class $\cC$. It follows that $C_1$ and $C_2$ are similar. 
\end{lemma}
\begin{proof}
As explained in the proof of the cylinder deformation theorem Wright~\cite[Theorem 1.1]{Wcyl}, by appealing to Smillie-Weiss~\cite{SW}, it is possible to use horocycle flow in the horizontal direction and arbitrarily small perturbations to start with the surface $(Y, \eta)$ with two simple horizontal cylinders in $\cC$ and end with a surface $(X, \omega)$ in $\M$ that is horizontally periodic with the twist space coinciding with the cylinder preserving space and so that $C_1$ and $C_2$ persist on $(X, \omega)$ as simple cylinders congruent to the corresponding cylinders on $(Y, \eta)$. Since the process only involves hororcycle flow and arbitrarily small perturbations, we have by Lemma~\ref{L:EC} that $\cC$ remains an equivalence class on $(X, \omega)$. By Lemma~\ref{L:similar1}, $C_1$ and $C_2$ are similar.
\end{proof}

%
%

\section{Proof of Theorem~\ref{T:MWE} - Classification of Higher Rank Nonarithmetic Affine Invariant Submanifolds in $\mathcal{H}(6)$ }~\label{S:MWE}

Thoughout this section we will make the following assumption:

\begin{ass}\label{A:nonarithmetic}
$\M$ is a rank two nonarithmetic affine invariant submanifold in $\mathcal{H}(6)$
\end{ass}

\noindent Our goal is to show that there is only one such affine invariant submanifold.

\begin{lemma}\label{L:proto} 
There is a vertically and horizontally periodic translation surface in $\M$ with two vertical and horizontal equivalence classes $\{\cC_1, \cC_2 \}$ and $\{ \cV_1, \cV_2 \}$ respectively so that the following hold: 
\begin{enumerate}
\item $\cV_2$ is contained in $\cC_2$ 
\item $\cC_1$ is contained in $\cV_1$
\item $\cC_1$ contains a vertical saddle connection
\item $\cV_2$ contains a horizontal saddle connection
\end{enumerate}
\end{lemma}
\begin{proof}
By the prototype lemma in Apisa~\cite{Apisa-hyp}, there is a translation surface $(X_0, \omega_0)$ that satisfies the first two hypotheses. Use the standard shear to shear $\cC_1$ so that it contains a vertical saddle connection and call the resulting translation surface $(X_1, \omega_1)$. 

Vertically collapse $\cC_1$ (i.e. apply the standard shear to send the period of the vertical saddle connection in $\cC_1$ to zero while only changing imaginary parts of periods). Each component of the boundary translation surface contains $\cC_2$, which was not affected by the degeneration. In particular, since $\cC_2$ contains $\cV_2$, each component of the boundary translation surface contains a vertical cylinder. Since each component of the boundary translation surface is contained in a rank one orbit closure by Mirzakhani-Wright~\cite[Corollary 2.8]{MirWri} it is completely vertically periodic by Wright~\cite[Theorem 1.5]{Wcyl}. Since $\cC_1$ was vertically collapsed, it follows that $(X_1, \omega_1)$ remains vertically and horizontally periodic and satisfies properties $1$ through $3$. 

Now use the standard shear to shear $\cV_2$ so that it contains a horizontal saddle connection and call the resulting surface $(X_2, \omega_2)$. The argument just given shows that $(X_2, \omega_2)$ is the desired surface.
\end{proof}

\begin{rem}
By Wright~\cite[Theorem 1.9]{Wcyl}, since $\M$ is nonarithmetic each equivalence class must contain at least two cylinders. Since a translation surface in $\mathcal{H}(6)$ has at most four cylinders in a given direction, each equivalence class contains exactly two cylinders. 
\end{rem}

\begin{ass}\label{A:proto}
Let $(X, \omega)$ be a surface in $\M$ satisfying the conclusion of Lemma~\ref{L:proto}. 
\end{ass}

\begin{defn}
``Collapsing $\cC_1$" will mean applying the standard shear to $\cC_1$ to only alter the imaginary parts of periods while sending the period of the vertical saddle connection contained in $\cC_1$ to zero. We will use the phrase ``collapsing $\cV_2$" analogously. 
\end{defn}

\begin{lemma}\label{L:collapse}
Collapsing $\cV_2$ produces a boundary translation surface where each component belongs to a rank one rel one orbit closure in the boundary of $\M$. The rel deformation is a linear combination of the images on the boundary of the standard shears of the horizontal equivalence classes on $(X, \omega)$. By symmetry of hypotheses, the same holds when $\cC_1$ is collapsed.
\end{lemma}
\begin{proof}
Let $(Z, \zeta)$ be any component of the surface obtained by collapsing $\cC_1$ and let $\cN$ be the largest orbit closure in the boundary of $\M$ containing $(Z, \zeta)$. Since $(Z, \zeta)$ contains a vertical cylinder that previously belonged to $\cV_1$ - which intersected $\cC_1$ and $\cC_2$ on $(X, \omega)$ - it follows that $\cN$ contains tangent vectors, say $\sigma_1$ and $\sigma_2$, coming from the standard shears on $\cC_1$ and $\cC_2$. 

By Mirzakhani-Wright~\cite[Corollary 2.8]{MirWri}, $\cN$ is rank one, meaning that the projection of $T_{(Z, \zeta)} \cN$ to absolute cohomology is a two complex-dimensional symplectic vector space. Since $\sigma_1$ and $\sigma_2$ are supported on disjoint cylinders, they have intersection number zero. Therefore, there is a constant $c$ so that $c\sigma_1$ and $\sigma_2$ project to the same absolute cohomology class. In particular, this means that $\sigma_1 - c\sigma_2$ is a nonzero relative deformation. Consequently, $\cN$ has rel at least one, and since $\cN$ is three complex-dimensional $\cN$ must be a rank one rel one orbit closure as desired.
\end{proof}

\begin{lemma}\label{L:sc2}
There is a cylinder in $\cC_2$ that has a horizontal saddle connection attaching its top boundary to its bottom boundary. This saddle connection is necessarily contained in $\cV_2$. By symmetry, the analogous statement holds for $\cV_1$.
\end{lemma}
\begin{proof}
Suppose not to a contradiction. Let $C_2$ and $C_2'$ be the two cylinders in $\cC_2$. Since the cylinders in $\cV_2$ are contained in $\cC_2$, each cylinder in $\cV_2$ intersects the core curve of $C_2$ and $C_2'$ an equal number of times (since there are no saddle connections joining the top and bottom of a cylinder in $\cC_2$). This means that the ratio of the lengths of the core curves of the cylinders in $\cV_2$ is an integer or the reciprocal of an integer. In particular, by Wright~\cite[Theorem 1.9]{Wcyl} $\M$ is arithmetic, which contradicts Assumption~\ref{A:nonarithmetic} that $\M$ is nonarithmetic.

It remains to show that $\cV_2$ contains the saddle connection. Collapsing $\cV_2$ creates a boundary translation surface that has rel by supported on its horizontal cylinders by Lemma~\ref{L:collapse}. This is not possible if a horizontal cylinder in $\cC_2$ has a horizontal saddle connection joining its top and bottom. Therefore, $\cV_2$ must contain any such saddle connection. 
\end{proof}


\begin{defn}
If two cylinders are connected by a saddle connection, then we will say that they are adjacent. If they are connected by a saddle connection on their top and bottom boundaries, then we will say that they are two-sided adjacent.
\end{defn}

\begin{defn}
Define $\M_0$ to be the smallest affine invariant submanifold containing the translation surfaces in Figure~\ref{F:MWE} with $\lambda$ equal to the golden ratio, opposite sides identified, all angles integer multiples of $\frac{\pi}{2}$, and $x_i$ and $y_i$ arbitrary positive real numbers for $i = 1, 2$.
\end{defn}

\begin{lemma}\label{L:MWE1}
If collapsing $\cV_2$ disconnects the surface, then $\M = \M_0$. 
\end{lemma}
\begin{proof}
Let $(Y, \eta)$ be the disconnected boundary translation surface that arises from collapsing $\cV_2$. Each component of $(Y, \eta)$ contains at least one cylinder from $\cV_1$. Since $\cV_1$ contains exactly two cylinders and $(Y, \eta)$ has at least two components, it follows that $(Y, \eta)$ has exactly two components, each of which contains exactly one cylinder from $\cV_1$. This implies that the two cylinders in $\cV_1$ share no boundary on $(X, \omega)$. In particular, this means that the two cylinders in $\cV_1$ are disjoint and, in fact, never intersect the same horizontal cylinder. 

Since each component of $(Y, \eta)$ must contain a cylinder from $\cC_1$ we notice that the cylinders from $\cC_1$ must also be disjoint and adjacent to a unique cylinder in $\cC_2$. If there is not a cylinder in $\cV_2$ that passes through both cylinders in $\cC_2$, then $(X, \omega)$ would be disconnected. It follows that the cylinders in $\cC_2$ are two-sided adjacent. Since the cylinders in $\cC_2$ are no longer adjacent on $(Y, \eta)$ it follows that $\cV_2$ contains the three horizontal saddle connections that join the top boundary of a cylinder in $\cC_2$ to the bottom boundary of another (one saddle connection comes from Lemma~\ref{L:sc2} and the other two come since the two cylinders in $\cC_2$ are two-sided adjacent on $(X, \omega)$, but not on $(Y, \eta)$).

Since there are seven horizontal saddle connections on $(X, \omega)$, the cylinders in $\cC_1$ are disjoint and simple. The similar cylinder lemma (Lemma~\ref{L:SCL}) implies that the two cylinders in $\cC_1$ are similar and hence both contain a vertical saddle connection. Moreover, the core curve of each cylinder in $\cV_1$ must intersect the core curve of a cylinder in $\cC_1$ either zero or one time, see the left subfigure of Figure~\ref{F:shear-v1}.

\begin{figure}[h]
\centering
    \begin{tikzpicture}
		\draw (0,0) -- (0,2) -- (1,2) -- (1,0) -- (0,0);
		\draw[dotted] (0,1) -- (1,1);
		\node at (-.5, 1.5) {$C_1$}; \node at (-.5, .5) {$C_2$};
	\end{tikzpicture}
	\qquad 
	    \begin{tikzpicture}
		\draw (0,0) -- (.25, 1) -- (0, 2) -- (1,2) -- (1.25, 1) -- (1,0) -- (0,0);
		\draw[dotted] (.25,1) -- (1.25,1);
		\draw[dotted] (.25, 0) -- (.25, 2);
		\draw[dotted] (1,0) -- (1,2);
		\node at (-.5, 1.5) {$C_1$}; \node at (-.5, .5) {$C_2$};
	\end{tikzpicture}
	\caption{ The cylinder $V_1$ }
	\label{F:shear-v1}
\end{figure}


By applying the standard shear to shear $\cC_1$ and $\cC_2$ in opposite directions as in Figure~\ref{F:shear-v1}, we may perturb $(X, \omega)$ in a way to make the cylinders in $\cV_1$ are simple. The simple cylinder lemma (Lemma~\ref{L:SCL}) implies that the cylinders in $\cV_1$ are similar on the perturbed surface. Since the two cylinders in $\cC_1$ were similar too we see that the cylinder proportion theorem~\cite[Proposition 3.2]{NW} implies that the two cylinders in $\cV_1$ are also similar on $(X, \omega)$. We can now write down the surface explicitly; it is the translation surface shown in Figure~\ref{F:nearly-MWE} or its mirror image about the vertical.


\begin{figure}[h]
\centering
    \begin{tikzpicture}
		\draw (0,0) -- (0,2) -- (2,2) -- (2,3) -- (3,3) -- (3,4) -- (4,4) -- (4,2) -- (3,2) -- (3,1) -- (1,1) -- (1,0) -- (0,0);
		\draw[dotted] (0,1) -- (1,1) -- (1,2);
		\draw[dotted] (2,1) -- (2,2) -- (3,2) -- (3,3) -- (4,3);
		\draw[dashed] (2,1) -- (1,2); \draw[dashed] (3,1) -- (2,2);	
		\node at (-1.5, .5) {$\lambda y_1$};
		\node at (-1.5, 1.5) {$\lambda y_2$};
		\node at (-1.5, 2.5) {$y_2$};
		\node at (-1.5, 3.5) {$y_1$};
		\node at (.5, -.5) {$\lambda x_1$};
		\node at (3.5, -.5) {$x_1$};
		\node at (1.5, 2.3) {$a$}; \node at (2.5, 3.3) {$b$};
		\node at (1.5, .7) {$\sigma(a)$}; \node at (2.5, .7) {$\sigma(b)$};
	\end{tikzpicture}
	\caption{ The translation surface $(X, \omega)$ }
	\label{F:nearly-MWE}
\end{figure}
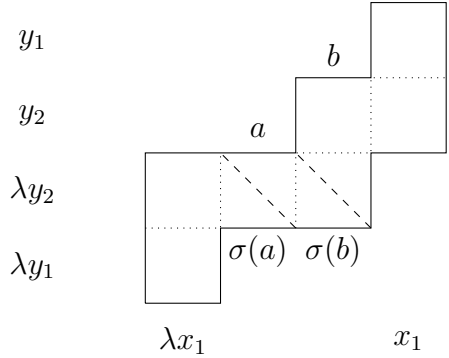 

In Figure~\ref{F:nearly-MWE}, opposite sides are identified unless labelled with $a, b, \sigma(a)$, or $\sigma(b)$. We let $\sigma$ be a permutation of $\{a, b\}$ that indicates how sides are identified. All other labels indicate lengths with $x_1, x_2, y_1, y_2$, and $\lambda$ being real numbers. By the cylinder proportion theorem~\cite[Proposition 3.2]{NW}, $\lambda > 1$. 

If $\sigma(b) = a$ then we see that one cylinder in $\cC_2$ contains a diagonal cylinder that does not intersect $\cV_1$. Since the other cylinder in $\cC_2$ cannot contain such a cylinder this violates the cylinder proportion theorem~\cite[Proposition 3.2]{NW} . Therefore, $\sigma(a) = a$ and $\sigma(b) = b$.  It follows that $(X, \omega)$ is the translation surface shown in Figure~\ref{F:MWE} or its mirror image about the vertical (for some choice of $x_1, x_2, y_1$, and $y_2$). The fact that $\M$ contains all surfaces in Figure~\ref{F:MWE} for arbitrary positive real $x_i$ and $y_i$ ($i=1,2$) follows from the fact that the standard shears in $\cC_i$ and $\cV_i$ remain in $\M$; this holds by  the cylinder deformation theorem~\cite[Theorem 1.1]{Wcyl}.

\begin{figure}[h]
\centering
    \begin{tikzpicture}
		\draw (0,0) -- (0,2) -- (2,2) -- (2,3) -- (3,3) -- (3,4) -- (4,4) -- (4,2) -- (3,2) -- (3,1) -- (1,1) -- (1,0) -- (0,0);
		\draw[dotted] (0,1) -- (1,1) -- (1,2);
		\draw[dotted] (2,1) -- (2,2) -- (3,2) -- (3,3) -- (4,3);	
		\node at (-1.5, .5) {$\lambda y_1$};
		\node at (-1.5, 1.5) {$\lambda y_2$};
		\node at (-1.5, 2.5) {$y_2$};
		\node at (-1.5, 3.5) {$y_1$};
		\node at (.5, -.5) {$\lambda x_1$};
		\node at (1.5, -1) {$(\lambda -1 ) x_2$};
		\node at (2.5, -.5) {$x_2$};
		\node at (3.5, -.5) {$x_1$};
	\end{tikzpicture}
	\caption{A surface in $\M_0$}
	\label{F:MWE}
\end{figure} 

Collapsing $\cC_1$ produces a boundary translation surface $(Z, \zeta)$ in $\mathcal{H}(1,1,0)$. By Lemma~\ref{L:collapse}, $(Z, \zeta)$ is contained in a three dimensional orbit closure $\cN$ in the boundary of $\cM$ and the relative deformation alters all four vertical cylinders. Therefore, $\cN$ is a three-dimensional nonarithmetic eigenform locus in $\mathcal{H}(1,1)$ with a non-Weierstrass periodic point marked. By Theorem~\ref{T2}, $\cN$ is the golden eigenform locus with a golden point marked and since $\lambda > 1$, $\lambda$ is the golden ratio. 

The affine invariant submanifold $\M_0$ is defined precisely as the smallest affine invariant submanifold containing the translation surfaces in Figure~\ref{F:MWE} with $\lambda$ the golden ratio and $x_i$ and $y_i$ arbitrary positive real numbers for $i = 1,2$. The four cohomology classes that correspond to changing $x_1$, $x_2$, $y_1$, or $y_2$ span a four complex-dimensional subspace of absolute cohomology. Since the tangent space of $\M$ is four complex-dimensional, it follows that these cohomology classes span the tangent space to $\M$ at $(X, \omega)$. Since these classes are contained in $\M_0$, it follows that $\M$ coincides with $\M_0$.

Recall finally that we showed that $(X, \omega)$ is the translation surface shown in Figure~\ref{F:MWE} up to taking a mirror image in the vertical. However, the mirror image of the translation surface in Figure~\ref{F:MWE} about the vertical is identical to its rotation by $\pi$ and hence in this case $\M$ is still equal to $\M_0$.
\end{proof}

In light of Lemma~\ref{L:MWE1} we make the following assumption:

\begin{ass}\label{A:connected}
The surface $(X, \omega)$ is not disconnected when $\cC_1$ or $\cV_2$ collapses.
\end{ass}

Recall that given a horizontally periodic surface the cylinder graph is the directed graph whose vertices are cylinders and that has a directed edge from cylinder $A$ to cylinder $B$ if the top boundary of $A$ borders the bottom boundary of $B$. The cylinder graph is always strongly connected. 

\begin{lemma}\label{L:at-most3}
There are at most three horizontal saddle connections connecting cylinders in $\cC_2$ to cylinders in $\cC_2$.
\end{lemma}
\begin{proof}
Suppose not to a contradiction. For a horizontally periodic translation surface in $\mathcal{H}(6)$ there are seven horizontal saddle connections and at least four connect cylinders in $\cC_2$ to other cylinders in $\cC_2$. Since only three saddle connections remain and since the cylinder graph must be strongly connected, there must be a directed edge from a cylinder in $\cC_2$ to a cylinder $C_1 \in \cC_1$; a directed edge from $C_1$ to a distinct cylinder $C_1' \in \cC_1$; and a directed edge from $C_1'$ to some cylinder in $\cC_2$. However, this configuration implies that there is a marked point on the boundary of $C_1$ and $C_1'$, which is not possible. Therefore, we have a contradiction.
\end{proof}

\begin{lemma}\label{L:two-sc}
There are two horizontal saddle connections contained in $\cV_2$. By symmetry, $\cC_1$ contains two vertical saddle connections. 
\end{lemma}
\begin{proof}
We divide the proposition into two cases.

\noindent \textbf{Case 1: $\cC_2$ is not two-sided adjacent}

If $\cC_2$ is not two-sided adjacent, then each cylinder in $\cV_2$ is contained in a unique cylinder in $\cC_2$. 

If $\cV_2$ contains three horizontal saddle connections, then the (horizontal) cylinder graph must contain three loops - directed edges that begin and end at the same vertex - based at vertices in $\cC_2$. Ignoring these edges we see that the remaining four edges must form a strongly connected directed graph on the four vertices. This is only possible if these edges form a directed four-cycle. This implies that the two cylinders in $\cC_1$ are homologous, which is impossible in a minimal stratum. Alternatively, the two cylinders in $\cC_1$ must have core curves of identical length, which yields the contradiction that $\M$ is arithmetic by Wright~\cite[Theorem 1.9]{Wcyl}. 

Therefore, $\cV_2$ must contain at most two horizontal saddle connections by Lemma~\ref{L:at-most3}. Since the two cylinders in $\cV_2$ are simple and disjoint they must be similar by the similar cylinder lemma (Lemma~\ref{L:SCL}). Since $\cV_2$ contains at least one horizontal saddle connection by Assumption~\ref{A:proto}, it must contain exactly two horizontal saddle connections.


\noindent \textbf{Case 2: $\cC_2$ is two-sided adjacent}

By Assumption~\ref{A:proto}, $\cV_2$ contains at least one horizontal saddle connection and by Lemma~\ref{L:at-most3} it contains at most three. 

If $\cV_2$ contains exactly one, then collapsing $\cV_2$ leaves $\cC_2$ two-sided adjacent. However, by Lemma~\ref{L:collapse} the boundary translation surface formed by collapsing $\cV_2$ has a rel deformation that is a linear combination of the standard shears on the horizontal cylinders. However, there is a closed curve $\gamma$ that is contained in $\cC_2$ (since $\cC_2$ is two-sided adjacent) and whose period is changed by standard shear on $\cC_2$. This contradicts the definition of a rel deformation. 

If $\cV_2$ contains exactly three horizontal saddle connections, then collapsing it produces a connected (by Assumption~\ref{A:connected}) translation surface with four horizontal cylinders. The horizontal cylinder graph contains four vertices and four edges and therefore is a directed four-cycle. This corresponds to a horizontally periodic flat torus with four marked points. On the boundary translation surface all cylinders have core curves of equal length. Since the lengths of the core curves of the cylinders in $\cC_1$ were unaffected by collapsing $\cV_2$, it follows that they had identical lengths on $(X, \omega)$ and hence $\M$ is arithmetic by Wright~\cite[Theorem 1.9]{Wcyl}; contradicting Assumption~\ref{A:nonarithmetic} which states that $\M$ is nonarithmetic. Therefore, $\cV_2$ contains exactly two horizontal saddle connections.

\end{proof}

\begin{lemma}\label{L:golden-collapse}
Collapsing $\cV_2$ yields a surface in the golden eigenform locus with the golden point marked. The same statement holds for $\cC_1$ by symmetry. 
\end{lemma}
\begin{proof}
Collapse $\cV_2$ and let $(Y, \eta)$ be the genus $g$ boundary translation surface. Let $\cN$ be the largest connected orbit closure in the boundary of $\M$ containing $(Y, \eta)$. By Assumption~\ref{A:connected}, $(Y, \eta)$ is connected and contains four horizontal cylinders. By Lemma~\ref{L:two-sc} it contains five horizontal saddle connections. By Lemma~\ref{L:collapse}, a linear combination of the standard shears of the two horizontal equivalence classes of cylinders on $(X, \omega)$ becomes a rel deformation on $(Y, \eta)$. In particular, this means that the number of zeros and marked points, $s$, on $(Y, \eta)$ is at least two. 

The horizontal saddle connections on a horizontally periodic translation surface is $2g + s - 2$, which in this case is five. Since $s \geq 2$, we have that $g = 1, 2$. If $g = 1$, then all horizontal cylinders have core curves of equal length. In particular, the cylinders in $\cC_1$ on $(X, \omega)$ have core curves of equal length, which implies that $\M$ is arithmetic by Wright~\cite[Theorem 1.9]{Wcyl}. This contradicts Assumption~\ref{A:nonarithmetic}. So $g = 2$ and $s = 3$.

Since the lengths of the core curves of the cylinders in $\cC_1$ is unchanged as $\cV_2$ collapses, the ratio of these lengths remains nonarithmetic on $(Y, \eta)$. Since $\cN$ is a rank one rel one orbit closure by Lemma~\ref{L:collapse}, $\cN$ is nonarithmetic and belongs to either $\mathcal{H}(2,0,0)$ or $\mathcal{H}(1,1,0)$. Since $\cN$ is nonarithmetic, the projection that sends $\cN$ to $\mathcal{H}(2)$ or $\mathcal{H}(1,1)$ by forgetting marked points is contained in a nonarithmetic eigenform locus $\cN'$. 


Suppose to a contradiction that $(Y, \eta)$ belongs to $\mathcal{H}(2,0,0)$.  Since $\cN$ is nonarithmetic, the projection to $\mathcal{H}(2)$ that forgets marked points has as its image a nonarithmetic Teichm\"uller curve. By Apisa-Wright~\cite[Theorem 1.3]{Apisa-Wright}, the two marked points are either (1) a periodic point and an unconstrained point, or (2) two points exchanged by the hyperelliptic involution. In the first case, the relative deformation is moving the unconstrained point (which involves altering two cylinders on $(Y, \eta)$) while fixing the rest of the surface. In the second, the relative deformation is moving one marked point and having its image under the hyperelliptic involution move in the opposite direction (which involves altering three cylinders on $(Y, \eta)$ since a marked point and its image under the hyperelliptic involution are contained in the same horizontal cylinder on the unmarked surface) while fixing the rest of the surface. However, the relative deformation is a linear combination of the two standard shears of the horizontal equivalence classes on $(X, \omega)$ and hence must involve all four horizontal cylinders. This is a contradiction. 

Therefore, $(Y, \eta)$ belongs to $\mathcal{H}(1,1,0)$. As before, $\cN'$ is either the decagon locus or a nonarithmetic eigenform locus. If $\cN'$ is the decagon locus, then the marked point is unconstrained and the relative deformation is simply moving the unconstrained point around the surface. This deformation involves at most two cylinders, which is a contradiction as before. 

Suppose now that $\cN'$ is a nonarithmetic eigenform locus in $\mathcal{H}(1,1)$. The marked point must be a periodic point. If the marked point is a Weierstrass point, then there are two congruent cylinders $C_1$ and $C_2$ on $(Y, \eta)$ that are generically congruent on $\cN$. Since these cylinders have generically equal heights in $\cN$ they must be the image of an equivalence class of horizontal cylinders on $(X, \omega)$. However, since they have generically equal lengths of core curves, $\M$ is arithmetic by Wright~\cite[Theorem 1.9]{Wcyl}, which contradicts Assumption~\ref{A:nonarithmetic}. Therefore, the marked point is a periodic point that is not a Weierstrass point. This implies that $\cN'$ contains surfaces in the golden eigenform locus with the golden point marked.

\end{proof}

Once $\cV_2$ is collapsed, Lemma~\ref{L:golden-collapse}  implies that the resulting boundary translation surface is the one shown in Figure~\ref{F:cyl2} up to rotation by $\pi$. The figure is only accurate in how it illustrates how horizontal cylinders are attached to one another. In Figure~\ref{F:cyl2}, opposite sides are identified, and the cylinders in each equivalence class are labelled.

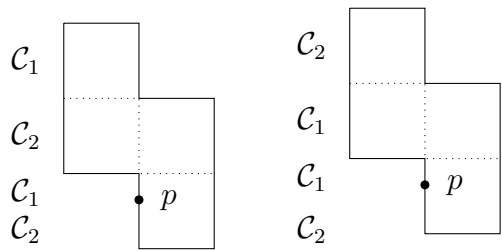
\begin{figure}[h]
\centering
	\begin{tikzpicture}
		\draw (0,1) -- (1,1) -- (1,0) -- (2,0) -- (2,-2) -- (1,-2) -- (1,-1) -- (0,-1) -- (0,1);
		\draw[dotted] (0,0) -- (1,0) -- (1,-1) -- (2,-1);
		\node at (-.5, .5) {$\cC_1$};
		\node at (-.5, -.5) {$\cC_2$};
		\node at (-.5, -1.25) {$\cC_1$};
		\node at (-.5, -1.75) {$\cC_2$};
		\draw[black, fill] (1,-1.35) circle[radius=1.5pt];\node at (1.4, -1.35) {$p$}; 
	\end{tikzpicture}
	\qquad 
		\begin{tikzpicture}
		\draw (0,1) -- (1,1) -- (1,0) -- (2,0) -- (2,-2) -- (1,-2) -- (1,-1) -- (0,-1) -- (0,1);
		\draw[dotted] (0,0) -- (1,0) -- (1,-1) -- (2,-1);
		\node at (-.5, .5) {$\cC_2$};
		\node at (-.5, -.5) {$\cC_1$};
		\node at (-.5, -1.25) {$\cC_1$};
		\node at (-.5, -1.95) {$\cC_2$};
		\draw[black, fill] (1,-1.35) circle[radius=1.5pt];\node at (1.4, -1.35) {$p$}; 
	\end{tikzpicture}
\caption{The boundary translation surface $(Y, \eta)$}
\label{F:cyl2}
\end{figure}

\begin{lemma}\label{L:two-disjoint}
Either $\cC_1$ or $\cV_2$ contains two disjoint simple cylinders.
\end{lemma}
\begin{proof}
Let $(Y, \eta)$ be the translation surface obtained by collapsing $\cV_2$ on $(X, \omega)$. We divide the proof into two cases. 

\noindent \textbf{Case 1: The two cylinders in $\cC_2$ are adjacent to each other }

If the two cylinders in $\cC_2$ are two-sided adjacent, then there are at least three horizontal saddle connections on the boundary of $\cC_2$ that connects a cylinder in $\cC_2$ to a cylinder in $\cC_2$ (there are two such saddle connections joining the two distinct cylinders and then one more joining a cylinder in $\cC_2$ to itself by Lemma~\ref{L:sc2}). If the two cylinders are one-sided adjacent then there are still at least three horizontal saddle connections on their boundary (there is one saddle connection joining the two distinct cylinders and then two saddle connections in $\cV_2$ that connect the boundary of a cylinder to itself). By Lemma~\ref{L:at-most3}, there are exactly three horizontal saddle connections on the boundary of $\cC_2$. By Lemma~\ref{L:two-sc}, there are only two horizontal saddle connections contained in $\cV_2$. Therefore, the two cylinders in $\cC_2$ remain adjacent on the boundary translation surface $(Y, \eta)$. In particular, $(Y, \eta)$ must be the rightmost figure in Figure~\ref{F:cyl2}. We see that the two cylinders in $\cC_1$ are disjoint and simple on $(X, \omega)$.

\noindent \textbf{Case 2: The two cylinders in $\cC_2$ are not adjacent to each other }

As we proved in Lemma~\ref{L:two-sc} Case 1, if the two cylinders in $\cC_2$ are not adjacent, then there are exactly two horizontal saddle connections that connect two cylinders in $\cC_2$. This shows that $\cV_2$ contains two disjoint simple cylinders. 

\end{proof}

Without loss of generality, perhaps after rotating $(X, \omega)$ by $\frac{\pi}{2}$ we make the following assumption: 

\begin{ass}\label{A:v-simple}
Assume that $\cV_2$ contains two disjoint simple cylinders
\end{ass}

This is the first assumption that we have made that is not symmetric in how it treats the horizontal and vertical direction. 

\begin{proof}[Proof of Theorem~\ref{T:MWE}]
We will begin by showing that $\M$ coincides with $\M_0$. By Lemma~\ref{L:golden-collapse} the boundary translation surface $(Y, \eta)$ formed by collapsing $\cV_2$ is the golden eigenform locus with the golden point marked (see Figure~\ref{F:cyl2}, which is only an accurate up to rotation by $\pi$, and only represents horizontal cylinders adjacency). Since $\cV_2$ contains two disjoint simple cylinders, these cylinders become vertical saddle connections on $(Y, \eta)$ that are contained in $\cC_2$

\noindent \textbf{Case 1: The cylinders in $\cC_2$ are not adjacent}

This implies that $(Y, \eta)$ is the rightmost figure in Figure~\ref{F:cyl2}. Since both cylinders in $\cC_2$ on $(Y, \eta)$ contain a vertical saddle connection we see that both cylinders in $\cC_2$ on $(Y, \eta)$ are as drawn in Figure~\ref{F:cyl2}. Moreover, apply the standard shear to the cylinders in $\cC_1$ so that $(Y, \eta)$ is, after deforming the cylinders in $\cC_1$, exactly as drawn in Figure~\ref{F:final-case1}. The two vertical saddle connections coming from the image of the core curves of the cylinders in $\cV_2$ are labelled $s_1$ and $s_2$.

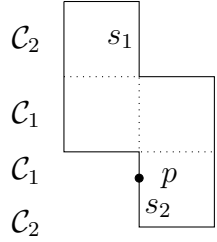
\begin{figure}[h]
\centering
	\qquad 
		\begin{tikzpicture}
		\draw (0,1) -- (1,1) -- (1,0) -- (2,0) -- (2,-2) -- (1,-2) -- (1,-1) -- (0,-1) -- (0,1);
		\draw[dotted] (0,0) -- (1,0) -- (1,-1) -- (2,-1);
		\node at (-.5, .5) {$\cC_2$};
		\node at (-.5, -.5) {$\cC_1$};
		\node at (-.5, -1.25) {$\cC_1$};
		\node at (-.5, -1.95) {$\cC_2$};
		\node at (1.25, -1.75) {$s_2$}; \node at (.75, .5) {$s_1$}; 
		\draw[black, fill] (1,-1.35) circle[radius=1.5pt];\node at (1.4, -1.35) {$p$}; 
	\end{tikzpicture}
\caption{The boundary translation surface $(Y, \eta)$ after applying the standard shear to $\cC_1$}
\label{F:final-case1}
\end{figure}

Opening up the collapsed simple cylinders shows us that the orbit closure contained the surface in Figure~\ref{F:MWE2} up to rotation by $\pi$. Opposite sides of the polygon are identified and labels indicate length where $x_i$ and $y_i$ are arbitrary positive real numbers for $i = 1, 2$, and $\phi$ is the golden ratio. 

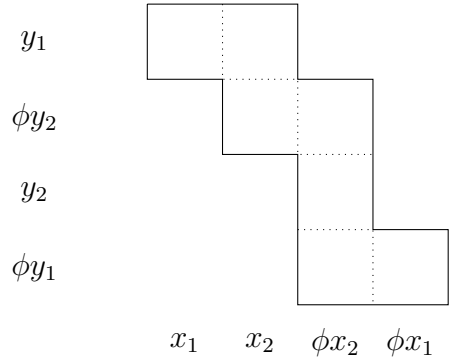
\begin{figure}[h]
\centering
    \begin{tikzpicture}
		\draw (0, 4) -- (2, 4) -- (2,3) -- (3,3) -- (3, 1) -- (4, 1) -- (4,0) -- (2,0) -- (2,2) -- (1,2) -- (1,3) -- (0,3) -- (0,4);
		\draw[dotted] (1,4) -- (1,3) -- (2,3) -- (2,2) -- (3,2);
		\draw[dotted] (2,1) -- (3,1) -- (3,0);
		\node at (-1.5, .5) {$\phi y_1$};
		\node at (-1.5, 1.5) {$y_2$};
		\node at (-1.5, 2.5) {$\phi y_2$};
		\node at (-1.5, 3.5) {$y_1$};
		\node at (.5, -.5) {$x_1$};
		\node at (1.5, -.5) {$x_2$};
		\node at (2.5, -.5) {$\phi x_2$};
		\node at (3.5, -.5) {$\phi x_1$};
	\end{tikzpicture}
	\caption{The surface in Figure~\ref{F:MWE} rotated}
	\label{F:MWE2}
\end{figure} 

This surface is the one in Figure~\ref{F:MWE} rotated by $\frac{\pi}{2}$ and so the orbit closure of $\M$ is $\M_0$ as desired.

\noindent \textbf{Case 2: The cylinders in $\cC_2$ are adjacent}

As we argued in Lemma~\ref{L:two-disjoint} Case 1, after $\cV_2$ collapses we must be in the leftmost figure of Figure~\ref{F:cyl2}. The simple cylinders in $\cV_2$ are now two vertical saddle connections in the cylinders labeled $\cC_2$. By Assumption~\ref{A:proto}, $\cC_1$ contains a vertical saddle connection. Since $\cC_1$ contains two disjoint simple cylinders they must be similar by the similar cylinder lemma (Lemma~\ref{L:SCL}) and so both contain a vertical saddle connection. Therefore, the surface drawn in Figure~\ref{F:cyl2} is correct up to determining how the cylinders in $\cC_2$ may be sheared to contain two vertical saddle connections. The four options are shown in Figure~\ref{F:four-options}.  

\begin{figure}[H]
    \begin{subfigure}[b]{0.25\textwidth}
        \centering
        \resizebox{\linewidth}{!}{\begin{tikzpicture}
        		\draw (0,0) -- (0,3.62) -- (1,3.62) -- (1,2.62) -- (2.62,2.62) -- (2.62,-3.23) -- (1,-3.23) -- (1,0) -- (0,0);
		\draw[dotted] (0,2.62) -- (1,2.62) -- (1,0) -- (2.62,0);
		\node at (-.5, 3.12) {$\cC_1$};
		\node at (-.5, 1.31) {$\cC_2$};
		\node at (-.5, -.81) {$\cC_2$};
		\node at (-.5, -2.43) {$\cC_1$};
		\node at (1.25, -.81) {$s_2$}; \node at (.25, 1.31) {$s_1$}; \node at (1.25, 1.31) {$s_1'$};
		\draw[black, fill] (1,-1.62) circle[radius=1.5pt];
	\end{tikzpicture}
        }
    \end{subfigure}
    \qquad
    \begin{subfigure}[b]{.25\textwidth}
        \centering
        \resizebox{\linewidth}{!}{\begin{tikzpicture}
        		\draw (0,0) -- (1,2.62) -- (1,3.62) -- (2,3.62) -- (2, 2.62) -- (3.62,2.62) -- (2,-1.62) -- (2,-3.23) -- (.38,-3.23) -- (.38,-1.62) -- (1,0) -- (0,0);
		\draw[dotted] (1,2.62) -- (2, 2.62); \draw[dotted] (1,0) -- (2.62,0);
		\draw[dashed] (1,0) -- (1, 2.62); \draw[dashed] (2, -1.62) -- (2, 2.62);
		\node at (2.25, 1.31) {$s_2$}; \node at (1.25, 1.31) {$s_1$};
		\draw[black, fill] (.38,-1.62) circle[radius=1.5pt];
	\end{tikzpicture}
        }
    \end{subfigure}
        \qquad
            \begin{subfigure}[b]{.3\textwidth}
        \centering
        \resizebox{\linewidth}{!}{\begin{tikzpicture}
        		\draw (0,0) -- (1.62,2.62) -- (1.62,3.62) -- (2.62,3.62) -- (2.62, 2.62) -- (4.24,2.62) -- (1.62,-1.62) -- (1.62,-3.23) -- (0,-3.23) -- (0,-1.62) -- (1,0) -- (0,0);
		\draw[dotted] (1.62,2.62) -- (2.62, 2.62); \draw[dotted] (1,0) -- (2.62,0);
		\draw[dashed] (2.62,0) -- (2.62, 2.62); \draw[dashed] (1.62, -1.62) -- (1.62, 2.62);
		\node at (2.25, 1.31) {$s_1$}; \node at (1.25, 1.31) {$s_2$};
		\draw[black, fill] (0,-1.62) circle[radius=1.5pt];
	\end{tikzpicture}
        }
    \end{subfigure}
  \caption{Four options for vertical saddle connections on $(Y, \eta)$} 
\label{F:four-options}
\end{figure}
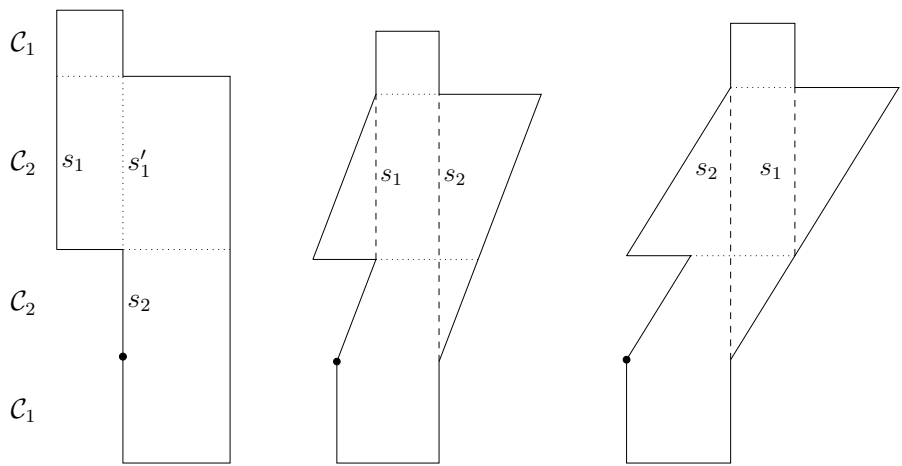

\noindent \textbf{Subcase 2a: Ruling out the leftmost figure in Figure~\ref{F:four-options}}

The first figure has two possible pairs of vertical saddle connections, either $\{s_1, s_2\}$ or $\{s_1', s_2\}$. The pair of cylinders $\{s_1, s_1'\}$ is not possible since this would imply that there is a cylinder in $\cC_2$ on $(X, \omega)$ that does not intersect $\cV_2$, which violates the cylinder proportion theorem~\cite[Proposition 3.2]{NW}. Opening up $\{s_1, s_2\}$ into simple cylinders produces the translation surface in Figure~\ref{F:two-possible}. Opposite sides are identified and labels are lengths where $\phi$ is the golden ratio and $x, y, z,$ and $w$ are arbitrary positive real numbers. Opening up $\{s_1', s_2\}$ yields the surface in Figure~\ref{F:two-possible} rotated by $\frac{\pi}{2}$. 

\begin{figure}[h]
\centering
\begin{tikzpicture}
		\draw (0,2) -- (0,3) -- (1,3) -- (1,4) -- (2,4) -- (2,3) -- (3,3) -- (3,2) -- (1,2) -- (1,0) -- (0,0) -- (0,1) -- (-1,1) -- (-1,2) -- (0,2);
		\draw[dotted] (0,1) -- (1,1);
		\draw[dotted] (0,2) -- (1,2);
		\draw[dotted] (0,3) -- (2,3);
		\draw[dotted] (1,2) -- (1,3);
		\draw[dotted] (0,1) -- (0,2);
		\draw[dotted] (2,2) -- (2,3);
		\node at (-1.5, .5) {$\phi y$};
		\node at (-1.5, 1.5) {$z$};
		\node at (-1.5, 2.5) {$\phi z$};
		\node at (-1.5, 3.5) {$y$};
		\node at (-.5, -.5) {$x$};
		\node at (.5, -.5) {$\phi w$};
		\node at (1.5, -.5) {$w$};
		\node at (2.5, -.5) {$\phi x$};
	\end{tikzpicture}
	\caption{The translation surface $(X, \omega)$}
	\label{F:two-possible}
\end{figure}
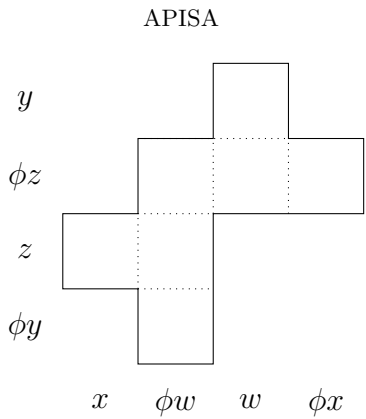

Choose $x$ and $w$ so that $\phi x < (\phi - 1) w$. In Figure~\ref{F:collapse-illustration}A, we identify two dashed saddle connections. We will apply the standard shear in $\cC_2$ so that these two saddle connections are vertical. Then we will apply the standard shear to preserve the real parts of all periods while sending the periods of these two saddle connections to zero. The dotted lines have been drawn in to keep track of where the zeros go in this degeneration. The result of the degeneration is shown in Figure~\ref{F:collapse-illustration}B. 

The surface $(Y, \eta)$ in Figure~\ref{F:collapse-illustration}B.  belongs to $\mathcal{H}(2,0,0)$ and the lengths of the edges are \[ |a| = \phi x \quad |b| = (\phi - 1) w - \phi x \quad |c| = x \quad |d| = w - x \]
Since $w$ and $x$ may be altered freely, as may the heights of the cylinders in $\cC_1$ we see that the boundary translation surface has orbit closure $\cN$ of dimension at least three. By Mirzakhani-Wright~\cite[Corollary 2.8]{MirWri}, $\cN$ has dimension exactly equal to three. Since $(Y, \eta)$ is a completely periodic surface whose horizontal cylinders have irrational ratio of lengths of core curves, $(Y, \eta)$ is an eigenform of real multiplication. Therefore, the only periodic points are Weierstrass points. By Apisa-Wright~\cite[Theorem 1.3]{Apisa-Wright}, the two marked points are either (1) a Weierstrass point and an unconstrained point, or (2) two points exchanged by the hyperelliptic involution. The two marked points are not exchanged by the hyperelliptic involution and neither is a Weierstrass point as long as $|a| \ne |b|$, which is generically the case. Therefore, we have a contradiction. 

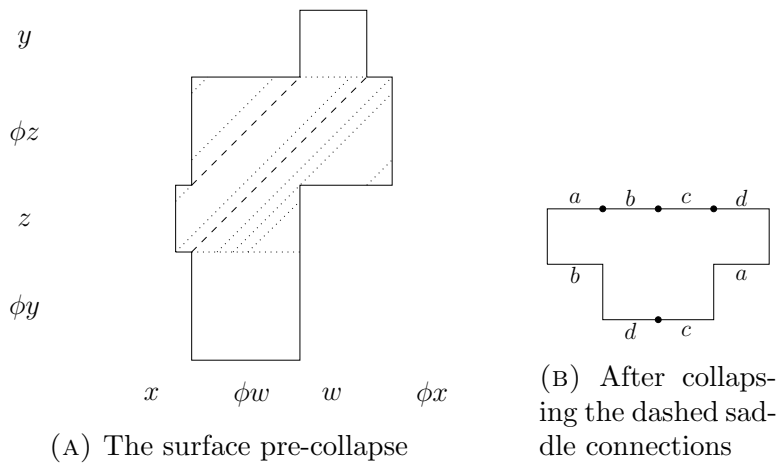
\begin{figure}[h]
       \begin{subfigure}[b]{0.5\textwidth}
        \centering
    	\resizebox{\linewidth}{!}{ \begin{tikzpicture}
		\draw (.76,1.62) -- (.76,2.62) -- (1,2.62) -- (1,4.24) -- (2.62,4.24) -- (2.62,5.24) -- (3.62,5.24) -- (3.62,4.24) -- (4,4.24) -- (4,2.62) -- (2.62,2.62) -- (2.62,0) -- (1,0) -- (1,1.62) -- (.76,1.62);
		\draw[dashed] (1, 2.62) -- (2.62,4.24);
		\draw[dashed] (1, 1.62) --  (3.62,4.24);
		\draw[dotted] (.76, 2.38) -- (1, 2.62);
		\draw[dotted] (1.62, 1.62) -- (4, 4);
		\draw[dotted] (1,4) -- (1.24, 4.24);
		\draw[dotted] (1.84, 1.62) --(2.62, 2.38);
		\draw[dotted] (.76, 1.62) -- (3.38, 4.24);
		\draw[dotted] (2.62, 4.24) -- (3.62, 4.24);
		\draw[dotted] (1, 1.62) -- (2.62, 1.62);
		\draw[dotted] (1.38,1.62) -- (4, 4.24);
		\draw[dotted] (3.62, 2.62) -- (4, 3); 
		\draw[dotted] (1,3) -- (2.24, 4.24); 
		\node at (-1.5, .8) {$\phi y$};
		\node at (-1.5, 2.1) {$z$};
		\node at (-1.5, 3.4) {$\phi z$};
		\node at (-1.5, 4.8) {$y$};
		\node at (.4, -.5) {$x$};
		\node at (1.9, -.5) {$\phi w$};
		\node at (3.1, -.5) {$w$};
		\node at (4.6, -.5) {$\phi x$};
	\end{tikzpicture}
	 }
     \caption{The surface pre-collapse}
    \end{subfigure}
    \qquad 
           \begin{subfigure}[b]{0.25\textwidth}
        \centering
    	\resizebox{\linewidth}{!}{ \begin{tikzpicture}
		\draw (0,0) -- (1,0) -- (1, -1) -- (3, -1) -- ( 3,0) -- ( 4,0) -- ( 4,1) -- (0,1) -- (0,0);
		\draw[black, fill] (1,1) circle[radius=1.5pt]; \node at (.5,1.2) {$a$}; \node at (.5, -.2) {$b$};
		\draw[black, fill] (2,1) circle[radius=1.5pt]; \node at (1.5, 1.2) {$b$}; \node at (1.5, -1.2) {$d$};
		\draw[black, fill] (3,1) circle[radius=1.5pt]; \node at (2.5, 1.2) {$c$}; \node at (2.5, -1.2) {$c$};
		\draw[black, fill] (2,-1) circle[radius=1.5pt]; \node at (3.5, 1.2) {$d$}; \node at (3.5, -.2) {$a$};
	\end{tikzpicture}
	 }
     \caption{After collapsing the dashed saddle connections}
    \end{subfigure}
	\caption{Degenerating $(X, \omega)$}
	\label{F:collapse-illustration}
\end{figure}	

\noindent \textbf{Subcase 2b: The right two figures in Figure~\ref{F:four-options}}

Slicing the right two figures in Figure~\ref{F:four-options} along $\{s_1, s_2\}$ and inserting a simple cylinder yields two figures that differ by a reflection across the vertical. Our argument will show that one and hence both resulting translation surfaces belong to $\M_0$. 

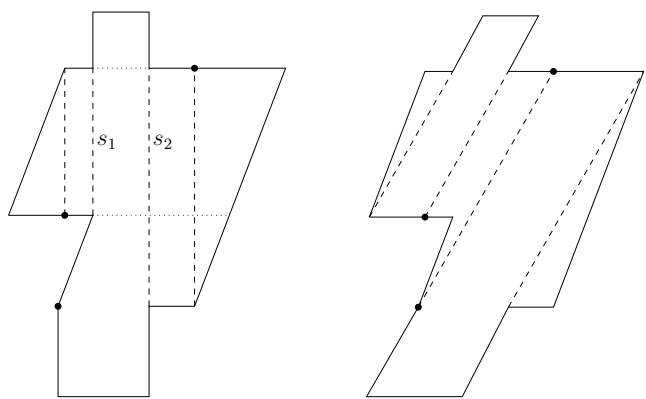
\begin{figure}[H]
    \begin{subfigure}[b]{.3\textwidth}
        \centering
        \resizebox{\linewidth}{!}{\begin{tikzpicture}
        		\draw (-.5,0) -- (.5, 2.62) -- (1,2.62) -- (1,3.62) -- (2,3.62) -- (2, 2.62) -- (2.81, 2.62)  -- (4.43,2.62) -- (2.81,-1.62) -- (2,-1.62) -- (2,-3.23) -- (.38,-3.23) -- (.38,-1.62) -- (1,0) -- (-.5,0);
		\draw[dotted] (1,2.62) -- (2, 2.62); \draw[dotted] (1,0) -- (3.43,0);
		\draw[dashed] (1,0) -- (1, 2.62); \draw[dashed] (2, -1.62) -- (2, 2.62);
		\draw[dashed] (2.81, -1.62) -- (2.81, 2.62); \draw[dashed] (.5, 0) -- (.5, 2.62);
		\node at (2.25, 1.31) {$s_2$}; \node at (1.25, 1.31) {$s_1$};
		\draw[black, fill] (.38,-1.62) circle[radius=1.5pt]; \draw[black, fill] (2.81,2.62) circle[radius=1.5pt];
		\draw[black, fill] (.5, 0) circle[radius=1.5pt];
	\end{tikzpicture}
        }
    \end{subfigure}
    \qquad
    \begin{subfigure}[b]{.3\textwidth}
        \centering
        \resizebox{\linewidth}{!}{\begin{tikzpicture}
        		\draw (-.5,0) -- (.5, 2.62) -- (1,2.62) -- (1.54,3.62) -- (2.54,3.62) -- (2, 2.62) -- (2.81, 2.62)  -- (4.43,2.62) -- (2.81,-1.62) -- (2,-1.62) -- (1.17,-3.23) -- (-.55,-3.23) -- (.38,-1.62) -- (1,0) -- (-.5,0);
		\draw[dashed] (-.5, 0) -- (1, 2.62); \draw[dashed] (.5, 0) -- (2, 2.62);
		\draw[dashed] (.38, -1.62) -- (2.81, 2.62); \draw[dashed] (2, -1.62) -- (4.43, 2.62);
		\draw[black, fill] (.38,-1.62) circle[radius=1.5pt]; \draw[black, fill] (2.81,2.62) circle[radius=1.5pt];
		\draw[black, fill] (.5, 0) circle[radius=1.5pt];
	\end{tikzpicture}
        }
    \end{subfigure}
  \caption{The surface produced by gluing in simple cylinders} 
\label{F:final-form}
\end{figure}

The surface we produce is the leftmost one in Figure~\ref{F:final-form}. We may shear the cylinders in $\cC_1$ to produce the rightmost figure in Figure~\ref{F:final-form}, which contains two equivalent cylinders $\cV_1'$ (outlined in dashed lines) that contain $\cC_1$ and two more equivalent cylinders $\cV_2'$ that are contained in $\cC_2$. Notice that Assumption~\ref{A:proto} and Assumption~\ref{A:connected} holds for this surface with the cylinder groups $\cC_1, \cC_2, \cV_1',$ and $\cV_2'$. We see too that $\cC_1$ contains two disjoint simple cylinders, which is Assumption~\ref{A:v-simple}. Since the two cylinders in $\cV_1$ are disjoint we are in Case 1 of this proof, which shows that the surface must belong to $\M_0$, as desired.

\noindent \textbf{Periodic Points in $\M$}

Now we will establish the second half the claim in Theorem~\ref{T:MWE}, namely that $\M$ has no periodic points. Suppose to a contradiction that $p$ is a periodic point on the translation surface $(Z, \zeta)$ shown in Figure~\ref{F:MWE}. By Apisa-Wright~\cite[Lemma 7.1]{Apisa-Wright} (which is a cleaner statement of a result in Apisa~\cite{Apisa-mp1}) if the marked point lies in the interior of a cylinder in $\cC_2$ then it is one of the solid points shown in Figure~\ref{F:MWE-final}. If it lies in the interior of a cylinder in $\cV_1$ then it is an open dot shown in Figure~\ref{F:MWE-final}.

\begin{figure}[h]
\centering
    \begin{tikzpicture}
		\draw (0,0) -- (0,2) -- (2,2) -- (2,3) -- (3,3) -- (3,4) -- (4,4) -- (4,2) -- (3,2) -- (3,1) -- (1,1) -- (1,0) -- (0,0);
		\draw[dotted] (0,1) -- (1,1) -- (1,2);
		\draw[dotted] (2,1) -- (2,2) -- (3,2) -- (3,3) -- (4,3);	
		\node at (-1.5, .5) {$\lambda y_1$};
		\node at (-1.5, 1.5) {$\lambda y_2$};
		\node at (-1.5, 2.5) {$y_2$};
		\node at (-1.5, 3.5) {$y_1$};
		\node at (.5, -.5) {$\lambda x_1$};
		\node at (1.5, -1) {$(\lambda -1 ) x_2$};
		\node at (2.5, -.5) {$x_2$};
		\node at (3.5, -.5) {$x_1$};
		\draw[black, fill] (.5, 1.5) circle[radius=1.5pt]; 
		\draw[black, fill] (2.25, 1.5) circle[radius=1.5pt];
		\draw[black, fill] (2.5, 2.5) circle[radius=1.5pt];
		\draw[black, fill] (3.5, 2.5) circle[radius=1.5pt];
		\draw[black] (.5, .5) circle[radius=1.5pt];
		\draw[black] (.5, 1.5) circle[radius=1.5pt];
		\draw[black] (3.5, 2.5) circle[radius=1.5pt];
		\draw[black] (3.5, 3.5) circle[radius=1.5pt];
	\end{tikzpicture}
	\caption{Potentially periodic points in $\M_0$}
	\label{F:MWE-final}
\end{figure}
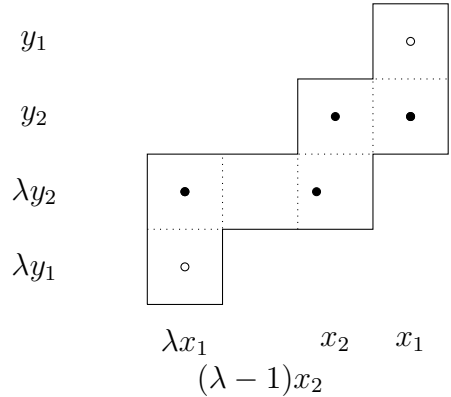 

Any point may be moved into the interior of either $\cV_1$ or $\cC_2$ by using the standard shear in a horizontal or vertical equivalence class to perform one complete Dehn twist in those cylinders. Using these Dehn twists we see that the bottom three points are all in the same $\GL(2, \R)$ orbit, as are the top three points. However, using the standard shear to perform one complete Dehn twist in the cylinders in $\cV_2$ we see that the two solid points in $\cV_2$ are sent to points that cannot be periodic points. Therefore, $\M$ contains no periodic points. 

\end{proof}

\bibliography{mybib}{}

\newcommand{\etalchar}[1]{$^{#1}$}
\providecommand{\bysame}{\leavevmode\hbox to3em{\hrulefill}\thinspace}
\providecommand{\MR}{\relax\ifhmode\unskip\space\fi MR }
\providecommand{\MRhref}[2]{%
  \href{http://www.ams.org/mathscinet-getitem?mr=#1}{#2}
}
\providecommand{\href}[2]{#2}
\begin{thebibliography}{BCG{\etalchar{+}}16}

\bibitem[AEM]{AEM}
Artur Avila, Alex Eskin, and Martin M\"oller, \emph{Symplectic and {I}sometric
  {$SL(2, \mathbb{R})$}-invariant subbundles of the {H}odge bundle}, preprint,
  arXiv 1209.2854 (2012).

\bibitem[Apia]{Apisa-mp1}
Paul Apisa, \emph{{GL(2,$\mathbb{R}$)}-invariant measures in marked strata:
  Generic marked points, {E}arle-{K}ra for strata, and illumination}, preprint,
  arXiv 1601.07894 (2016).

\bibitem[Apib]{Apisa-hyp}
\bysame, \emph{{GL(2,$\mathbb{R}$)} orbit closures in hyperelliptic components
  of strata}, preprint, arXiv:1508.05438 (2015).

\bibitem[AW17]{Apisa-Wright}
Paul Apisa and Alex Wright, \emph{Marked points on translation surfaces}, 2017,
  preprint, arXiv:1708.03411.

\bibitem[BCG{\etalchar{+}}16]{Backgammon-k}
Matt Bainbridge, Dawei Chen, Quentin Gendron, Samuel Grushevsky, and Martin
  Moeller, \emph{Strata of $k$-differentials}, 2016.

\bibitem[EM]{EM}
Alex Eskin and Maryam Mirzakhani, \emph{Invariant and stationary measures for
  the {$SL(2,\mathbb{R})$} action on moduli space}, preprint, arXiv 1302.3320
  (2013).

\bibitem[EMM15]{EMM}
Alex Eskin, Maryam Mirzakhani, and Amir Mohammadi, \emph{Isolation,
  equidistribution, and orbit closures for the {${\rm SL}(2,\Bbb R)$} action on
  moduli space}, Ann. of Math. (2) \textbf{182} (2015), no.~2, 673--721.
  \MR{3418528}

\bibitem[Fil16]{Fi1}
Simion Filip, \emph{Splitting mixed {H}odge structures over affine invariant
  manifolds}, Ann. of Math. (2) \textbf{183} (2016), no.~2, 681--713.
  \MR{3450485}

\bibitem[KM16]{MK16}
Abhinav Kumar and Ronen~E. Mukamel, \emph{Real multiplication through explicit
  correspondences}, 2016.

\bibitem[LMW16]{LMW}
Samuel Leli\`evre, Thierry Monteil, and Barak Weiss, \emph{Everything is
  illuminated}, Geom. Topol. \textbf{20} (2016), no.~3, 1737--1762.
  \MR{3523067}

\bibitem[LW15]{LW}
Samuel Leli\`evre and Barak Weiss, \emph{Translation surfaces with no convex
  presentation}, Geom. Funct. Anal. \textbf{25} (2015), no.~6, 1902--1936.
  \MR{3432160}

\bibitem[McM03]{Mc}
Curtis~T. McMullen, \emph{Billiards and {T}eichm\"uller curves on {H}ilbert
  modular surfaces}, J. Amer. Math. Soc. \textbf{16} (2003), no.~4, 857--885
  (electronic).

\bibitem[McM05]{McM:spin}
\bysame, \emph{Teichm\"uller curves in genus two: discriminant and spin}, Math.
  Ann. \textbf{333} (2005), no.~1, 87--130.

\bibitem[McM06]{Mc4}
\bysame, \emph{Teichm\"uller curves in genus two: torsion divisors and ratios
  of sines}, Invent. Math. \textbf{165} (2006), no.~3, 651--672.

\bibitem[McM07]{Mc5}
\bysame, \emph{Dynamics of {${\rm SL}_2(\Bbb R)$} over moduli space in genus
  two}, Ann. of Math. (2) \textbf{165} (2007), no.~2, 397--456.

\bibitem[M{\"o}l06]{M2}
Martin M{\"o}ller, \emph{Periodic points on {V}eech surfaces and the
  {M}ordell-{W}eil group over a {T}eichm\"uller curve}, Invent. Math.
  \textbf{165} (2006), no.~3, 633--649.

\bibitem[MW16]{MirWriRank}
Maryam Mirzakhani and Alex Wright, \emph{Full rank affine invariant
  submanifolds}, 2016, arXiv:1608.02147.

\bibitem[MW17]{MirWri}
Maryam Mirzakhani and Alex Wright, \emph{The boundary of an affine invariant
  submanifold}, Invent. Math. \textbf{209} (2017), no.~3, 927--984.
  \MR{3681397}

\bibitem[NW14]{NW}
Duc-Manh Nguyen and Alex Wright, \emph{Non-{V}eech surfaces in
  {$\mathcal{H}^{\rm hyp}(4)$} are generic}, Geom. Funct. Anal. \textbf{24}
  (2014), no.~4, 1316--1335. \MR{3248487}

\bibitem[SW10]{SW}
John Smillie and Barak Weiss, \emph{Characterizations of lattice surfaces},
  Invent. Math. \textbf{180} (2010), no.~3, 535--557.

\bibitem[vdG88]{vdG}
Gerard van~der Geer, \emph{Hilbert modular surfaces}, Ergebnisse der Mathematik
  und ihrer Grenzgebiete (3) [Results in Mathematics and Related Areas (3)],
  vol.~16, Springer-Verlag, Berlin, 1988. \MR{930101}

\bibitem[Wri15]{Wcyl}
Alex Wright, \emph{Cylinder deformations in orbit closures of translation
  surfaces}, Geom. Topol. \textbf{19} (2015), no.~1, 413--438. \MR{3318755}

\end{thebibliography}
\bibliographystyle{amsalpha}
\end{document}